\documentclass[a4paper, 11pt]{amsart}

\usepackage{fullpage}

\usepackage{amssymb}
\usepackage{amsmath,amssymb,tikz-cd}
\usepackage{subdepth}
\usepackage{tensor}
\usepackage{braket}
\usepackage{enumitem}

\numberwithin{equation}{section}

\newtheoremstyle{bfnote}%
{}{}%
{}{}%
{\bfseries}{.}%
{ }%
{\thmname{#1}\thmnumber{ #2}\thmnote{ (#3)}}

\theoremstyle{plain}
\newtheorem{theorem}{Theorem}[section]
\newtheorem{proposition}[theorem]{Proposition}
\newtheorem{lemma}[theorem]{Lemma}
\newtheorem{corollary}[theorem]{Corollary}

\theoremstyle{bfnote}
\newtheorem{condition}[theorem]{Condition}

\theoremstyle{definition}
\newtheorem{definition}[theorem]{Definition}
\newtheorem{assumption}[theorem]{Assumption}

\theoremstyle{remark}
\newtheorem{remark}[theorem]{Remark}
\newtheorem{example}[theorem]{Example}

\usepackage{xcolor}
\usepackage{mathtools}

\usepackage[colorlinks]{hyperref}
\hypersetup{
  allcolors = [rgb]{0,0,0.5}
}
\usepackage[nobysame,alphabetic,initials,msc-links]{amsrefs}

\DefineSimpleKey{bib}{how}

\renewcommand{\eprint}[1]{#1}
\renewcommand{\PrintReviews}[1]{} 
\BibSpec{misc}{%
  +{}{\PrintAuthors}  {author}
  +{,}{ \textit}      {title}
  +{,}{ }             {how}
  +{}{ \parenthesize} {date}
  +{,} { available at \eprint}        {eprint}
  +{,}{ available at \url}{url}
  +{,}{ }             {note}
  +{.}{}              {transition}
}

\mathchardef\mhyph="2D

\newcommand{\bA}{\mathbb{A}}

\newcommand{\bE}{\mathbb{E}}

\newcommand{\bN}{\mathbb{N}}

\newcommand{\bT}{\mathbb{T}}
\newcommand{\bZ}{\mathbb{Z}}

\newcommand{\cC}{\mathcal{C}}
\newcommand{\cD}{\mathcal{D}}
\newcommand{\cK}{\mathcal{K}}
\newcommand{\cL}{\mathcal{L}}
\newcommand{\cM}{\mathcal{M}}

\newcommand{\cZ}{\mathcal{Z}}

\newcommand{\id}{\mathrm{id}}
\newcommand{\reg}{\mathrm{reg}}
\newcommand{\DHR}{\mathrm{DHR}}
\newcommand{\opo}{\mathrm{mop}}
\newcommand{\tr}{\mathrm{tr}}
\newcommand{\Ad}{\mathrm{Ad}}
\newcommand{\fd}{\mathrm{fd}}
\newcommand{\CstAlg}{\mathrm{C}^*\mhyph\mathrm{alg}}
\newcommand{\supp}{\operatorname{supp}}
\newcommand{\Corr}{\operatorname{Corr}}
\newcommand{\Irr}{\operatorname{Irr}}
\newcommand{\Hom}{\operatorname{Hom}}
\newcommand{\End}{\operatorname{End}}

\newcommand{\Ind}{\operatorname{Ind}}
\newcommand{\Rep}{\operatorname{Rep}}
\newcommand{\Span}{\operatorname{Span}}

\newcommand{\rInd}{\mathrm{r}\mhyph\mathrm{Ind}}
\newcommand{\rI}{\mathrm{r}\mhyph\mathrm{I}}
\newcommand{\absv}[1]{\left|#1\right|}
\newcommand{\norm}[1]{\left\|#1\right\|}
\newcommand{\mathbbm}[1]{\mathbb{#1}}

\title{On the structure of DHR bimodules of abstract spin chains}
\author{Lucas Hataishi}
\address{University of Oxford}
\email{lucas.hataishi@gmail.com}
\author{David Jaklitsch}
\address{Universitetet i Oslo}
\email{dajak@math.uio.no}
\author{Corey Jones}
\address{North Carolina State University}
\email{cmjones6@ncsu.edu}
\author{Makoto Yamashita}
\address{Universitetet i Oslo}
\email{makotoy@math.uio.no}

\date{September 9, 2025, minor changes; April 8, 2025, r1}

\begin{document}

\begin{abstract}
Abstract spin chains axiomatize the structure of local observables on the 1D lattice which are invariant under a global symmetry, and arise at the physical boundary of 2+1D topologically ordered spin systems. In this paper, we study tensor categorical properties of DHR bimodules over abstract spin chains.
Assuming that the charge transporters generate the algebra of observables, we prove that the associated category has a structure of modular tensor category with respect to the natural braiding.
Under an additional assumption of algebraic Haag duality, this category becomes the Drinfeld center of the half-line fusion category.
\end{abstract}
\maketitle
\tableofcontents

\section{Introduction}

The theory of quantum spin chains provides an intriguing manifestation of \emph{discrete} quantum field theory, where ideas from the theory of quantum integrable systems, special functions, and operator algebras meet.
In this context, the algebra of local observables is modeled by the tensor product of finite-dimensional matrix algebras, which can be thought of as the kinematical component of a locally finite-dimensional 1+1D algebraic quantum field theory, but without a designated choice of Hilbert space representation or symmetry group~\cite{MR1441540}.
For such systems, tools from the theory of operator algebras turn out to be useful in understanding the algebraic structures behind localized excitations and their interaction with observables. 

Motivated by recent developments in quantum many-body physics, the third-named author introduced the notion of \emph{abstract spin chains}~\cite{2304.00068}.
These are formally defined as nets $A_{\bullet}$ of finite dimensional C$^*$-algebras over $\bZ$, viewed as a discrete metric space.
These objects mathematically characterize the algebras of local observables that emerge under symmetry constraints~\cites{PRXQuantum.4.020357, KawAnn, 2205.15243}, or as holographic duals of topologically ordered spin systems~\cite{2307.12552}.
Unlike concrete spin chains, abstract spin chains can exhibit \emph{algebraic entanglement}, meaning that the algebra of observables localized in some region does not necessarily factorize into the tensor product of algebras localized at points.
This leads to the emergence of categorical structures which provide new algebraic tools for studying topological spin systems through holography and dualities.

There are some obvious similarities to the local operator algebraic approach to the 1+1D chiral conformal field theory, based on nets of infinite-dimensional von Neumann algebras over the real line $\mathbbm{R}$ (or the compactified circle $S^1$)~\cite{Frolich-Gabbiani}.
However, in contrast to the more purely algebraic situation described above, conformal nets have full diffeomorphism group symmetries which lead to intricate interactions between the analytic structure of infinite-dimensional operator algebras and the underlying topological structure of space(time).
Nevertheless, there appear to be deeper structural similarities between the discrete and continuous settings.
In particular, if $A$ is an abstract spin chain, we can associate the braided tensor category of bimodules $\DHR(A_{\bullet})$~\cite{2304.00068}, in direct analogy with the braided category of superselection sectors of conformal nets~\cite{Frolich-Gabbiani}.
This structure is essential in the applications of abstract spin chains in 2+1D topological holography~\cites{inamura202321dsymmetrytopologicalorderlocalsymmetric, PhysRevB.107.155136, 2307.12552}, and the Kramers--Wannier type dualities for spin chains with (fusion categorical) symmetries~\cites{Aasen_2016, PRXQuantum.4.020357, 2304.00068}.

There are many general results about 1+1D chiral conformal field theories, and it is natural to ask which of these might have analogues in the setting of abstract spin chains.
Our goal in this paper is to investigate tensor categorical properties of $\DHR(A_{\bullet})$ beyond this first step, in parallel to foundational results of conformal field theory.
In particular, one motivating result is the modularity of the braided category associated with rational conformal nets~\cite{Kawahigashi-Longo-Muger}.
Based on the analogy with this theory, we can ask the following questions on abstract spin chains.
\begin{enumerate}[itemsep=1ex]
\item When is $\DHR(A_{\bullet})$ a fusion category?
\item When is $\DHR(A_{\bullet})$ modular?
\item When is $\DHR(A_{\bullet})$ Witt trivial, or in other words, when is $\DHR(A_{\bullet})$ a Drinfeld center of a fusion category?
\end{enumerate}
We provide natural sufficient conditions for all of these to hold.

In fact, the main examples of abstract spin chains are \emph{fusion spin chains}, which can be constructed from algebraic data arising from unitary fusion categories.
In these cases, the category of DHR bimodules are manifestly the Drinfeld center of the input fusion category, and such close coupling with general unitary fusion categories is one nice feature of abstract spin systems which makes them more accessible than the usual algebraic quantum field theories.
In view of our structural results, we can conclude that the fusion spin chains are (in some sense) universal examples of abstract spin chains up to categorical equivalence.

Let us turn to the main ideas behind our analysis.
We look at the inclusion of C$^*$-algebras $B_0 \subset A$, where $B_0$ is the subalgebra of observables supported on the half lines $(-\infty, -1]$ and $[0, \infty)$, and $A$ is the algebra of {\em quasi-local observables}.
The theory of subfactors suggest that we should look at the \emph{basic construction} $A \subset B_1$ associated with this inclusion, where $B_1$ is the algebra generated by $A$ and the Jones projection $e$, that satisfies $e a e = E(a) e$ for $a \in A$, where $E$ denotes a conditional expectation from $A$ to the subalgebra $B_0$.

A tensor categorical analogue of this construction is $A_\cC = \bigoplus_i U_i \otimes U_i^*$, where $(U_i)_i$ is a choice of representatives of the irreducible classes in some tensor category $\cC$.
This has a structure of algebra object in the Drinfeld center of (the ind-completion of) $\cC$, and plays an important role in the study of tensor categories.
Under this analogy, the projection to the trivial summand $U_0 \otimes U_0^*$ for the monoidal unit $U_0 = 1_\cC$ corresponds to the dual conditional expectation $\hat{E} \colon B_1 \to A$ characterized by $\hat{E}(e) = (\Ind E)^{-1} 1_A$, where $\Ind E$ is the \emph{index} of $E$ giving the relative size of $A$ over $B_0$.
This suggests that the finiteness of $\Ind E$ corresponds to the property of $\DHR(A_{\bullet})$ being a fusion category, which we substantiate in Corollary~\ref{cor:finite-index-implies-fusion-cat}.

Another key ingredient in our proof of modularity is a system of local operators called \emph{charge transporters}.
This is an analogue of unitary intertwiners that shows up in the theory of algebraic quantum field theory to relate the endomorphism realizations of the same superselection sector, and is a crucial ingredient in Müger's analysis of such theories~\cite{MR1721563}.
In the context of holographic topological order these have been called \emph{patch operators}~\cites{inamura202321dsymmetrytopologicalorderlocalsymmetric, PhysRevB.107.155136}, and have played a crucial role in the physicists' approach to accessing the information that is contained in the DHR category (which in the case of topological holography is the bulk topological order).

For a fusion category $\cC$ and its Drinfeld center $\cZ(\cC)$, the triviality of monodromy for an object $X \in \cZ(\cC)$ and the distinguished object $A_\cC$ implies that the underlying $\cC$-object for $X$ is a direct sum of copies of $1_\cC$, which can be taken as a first step for the proof of modularity for $\cZ(\cC)$.
The above analogy between $A_\cC$ and $B_1$ suggests that we should look at the monodromy for $X \in \DHR(A_{\bullet})$ and $B_1$ to understand the modularity question for $\DHR(A_{\bullet})$.
Indeed, the triviality for monodromy with $B_1$ implies that the charge transporters for $X$, with respect to excitations in the positive and negative regions, belong to the subalgebra $B_0$.
Expanding on this correspondence, we show that $\DHR(A_{\bullet})$ is modular when $B_0$ and the charge transporters for the DHR bimodules generate the observable algebra $A$ (Theorem~\ref{thm:triviality}).

We then look at the question of Witt triviality.
The construction of fusion spin chains suggests that $\DHR(A_{\bullet})$ should be the Drinfeld center of the bimodule category $\cC_-$ for the subalgebra $A_- \subset A$, which is generated by the observables on the negative half line $(-\infty, -1]$.
The structure of braiding on $\DHR(A_{\bullet})$ leads to a natural tensor functor $\DHR(A_{\bullet}) \to \cZ(\cC_-)$.
Now, looking at the correspondence between the central vectors into $A$-bimodules and the morphisms from the $A$ itself as a bimodule, together with a consideration on charge transporters, show that this functor is in fact full.
This result could be viewed as the main result of the paper, so we explicitly state a version of it here. 

\begin{theorem}
Let $A_{\bullet}$ be an abstract spin chain with charge transporter generation (see Condition~\ref{cond:reduced-net}). Then 

\begin{enumerate}
\item
If $A_{\bullet}$ is rational (Condition~\ref{cond:rationality}), then $\DHR(A_{\bullet})$ is a fusion category.
\item 
If in addition $A_{\bullet}$ is locally aligned (Condition~\ref{cond:B1-is-DHR}), then $\DHR(A_{\bullet})$ is unitary modular tensor category.
\item 
If $A_{\bullet}$ also satisfies the algebraic Haag duality (Condition~\ref{cond:Haag}), then $\DHR(A_{\bullet})$ is braided equivalent to $\mathcal{Z}(\mathcal{C}_{-})$, hence is Witt trivial. 
\end{enumerate}
\end{theorem}

We note that this last result \emph{contrasts} with the case of conformal nets, where the DHR category of superselection sectors is in general not Witt trivial, for example in the case of conformal nets built from loop groups~\cites{MR1645078,gui2020unbounded}.
However, even in the conformal net case, we still expect DHR categories to be Drinfeld centers of \emph{some} W$^*$-tensor category, which will generically fail to be fusion~\cite{MR3747830}.
Indeed our category $\mathcal{C}_{-}$ is very similar in spirit to the W$^*$-tensor category used in the above reference.

In addition, our result gives a holographic confirmation of expectations concerning the non-chirality of local topological order.
Indeed, local topological order essentially captures the class of topologically ordered Hamiltonians that have a commuting projector representation, e.g., toric code, quantum double models, and the Levin--Wen models.
It is largely expected to be the case that the resulting modular tensor category for this class of topologically ordered spin systems has no chiral anomaly, or in other words, is a Drinfeld center of some fusion category.

Let us summarize the structure of the paper.
In Section~\ref{sec:prelim}, we provide preliminary details on abstract spin chains.
In Section~\ref{sec:dist-alg-objs}, we recall some standard construction of algebras that will be crucial to our discussion in next sections.
In Section~\ref{sec:half-line-category}, we introduce the half-line category $\cC_-$, which serves as the candidate for witnessing the Witt triviality of $\DHR(A_{\bullet})$.
In Section~\ref{sec:mod-dri-cent}, we look at the modularity question of $\DHR(A_{\bullet})$.
In Section~\ref{sec:realiz-center}, we look at the Witt triviality of $\DHR(A_{\bullet})$.
Finally, in Section~\ref{sec:lattice-from-inv-bimod}, we give a construction of abstract spin chain generalizing that of fusion spin chains.
Our model is based on a variation of the Longo--Roberts recognition for a construction of bimodules from $2$-categories, and provides an example of abstract spin chains without any homogeneity for the underlying lattice.
In Appendix~\ref{sec:duality}, we look at the question of duality, and show that if an object $X \in \DHR(A_{\bullet})$ is dualizable as an $A$-correspondence, then its dual object also belongs to $\DHR(A_{\bullet})$.

\bigskip
\paragraph{Acknowledgments}
This project started at Topological Quantum Computation meeting at ICMS, Edinburgh, in October 2023.
We thank the organizers for setting up a productive environment that stimulated our discussion.
L.H. is supported by the Engineering and Physical Sciences Research Council (EP/X026647/1).
M.Y. is partially supported by The Research Council of Norway [project 300837].
D.J.\ and M.Y.\ are supported by The Research Council of Norway - project 324944. C.J. is supported by NSF Grant DMS- 2247202.

\section{Abstract spin chains}\label{sec:prelim}

In this section, we briefly review definitions and notations which we will make use of in the main body of work.

Let $\mathcal{I}$ denote the set of discrete intervals in $\bZ$, partially ordered with inclusion.
This is contained in the poset $\mathcal{P}(\bZ)$ of all subsets of $\bZ$ with inclusion.
For $F\subseteq \bZ$, we will denote its complement by $F^{c}$.
Furthermore, for any subset $F$, we will denote its $R$ neighborhood by $F^{+R}$.

Recall that any poset can naturally be viewed as a category, with a unique morphism $I\rightarrow J$ precisely when $I\le J$.
Let $\CstAlg_{\fd}$ denote the category of finite-dimensional C$^*$-algebras whose morphisms are unital inclusions.
Then we have the following definition.

\begin{definition}
An \emph{abstract spin chain} is a functor $A_{\bullet}\colon \mathcal{I}\rightarrow \CstAlg_{\fd}$ such that for any intervals $I\cap J=\varnothing$, $[A_{I},A_{J}]=0$ in $A_{K}$ for $I\cup J\subset K$.
\end{definition}

Associated to an abstract spin chain is the \emph{quasi-local} C$^*$-algebra $A = \lim_{\mathcal{I}} A_{I}$.
This is a unital AF C$^*$-algebra.
Since we have assumed our connecting maps are inclusions, we can identify each $A_{I}$ as a subalgebra of $A$, with $I\subseteq J$ implying $A_{I}\subseteq A_{J}$ as a C$^*$-subalgebra.

For any subset $F\subseteq \bZ$, we can define the subalgebra $A_{F}\subseteq A$ by
\[
A_{F}=\CstAlg \set{ x\in A_{G} | G\subseteq F}.
\]

We now introduce some useful properties abstract spin systems can have, which we view as ``regularity'' conditions.
Similar to the situation for conformal nets, there are many regularity conditions.
In the theorems below we will use combinations of these properties as hypotheses, so it will be convenient to give them names.
Many are extensions of familiar properties from the theory of conformal nets~\cite{Frolich-Gabbiani}.

\bigskip

\begin{condition}[Weak algebraic Haag duality]\label{cond:Haag}
There are $K, R\ge 0$ such that for every interval $I$ of length at least $K$, $A^{\prime}_{I^{c}}\cap A\subset A_{I^{+R}}$.
If we can choose $R=0$, we say $A$ satisfies \emph{algebraic Haag duality}.
\end{condition}

\begin{condition}[Covering property]\label{cond:generation}
There is an $L>0$ such that if $I,J$ are intervals whose intersection contains an interval of length $L$, then $A_{I\cup J}=A_{I}\vee A_{J}$.
\end{condition}

\begin{condition}[Strong simplicity]\label{cond:simplicity}
The quasi-local algebra $A$ is simple with unique trace.
\end{condition}

\begin{condition}\label{cond:B0-is-simple}
The centers of the half-line algebras $A_{(-\infty,a]}$ and $A_{[a,\infty)}$ are $1$-dimensional for all $a\in \bZ$.
\end{condition}

\subsection{DHR bimodules}

If $A$ is a C$^*$-algebra, recall that a correspondence over $A$ is a right Hilbert $A$ module $X$ together with a (non-degenerate) homomorphism from $A$ to the adjointable operators on $X$ (see~\cite{MR4419534} for a detailed exposition).
We usually think of a correspondence as an (algebraic) $A$-$A$ bimodule $X$ equipped with a right $A$-valued inner product $\braket{\xi|\eta}$ for $\xi, \eta \in X$, which is conjugate linear in first variable and right $A$-linear in the second.

The collection of all correspondences over $A$ forms a C$^*$-tensor category we call $\Corr(A)$, whose objects are correspondences over $A$ and morphisms are adjointable bimodule intertwiners.
The monoidal product is the relative tensor product, denoted
\[
X\boxtimes_{A}Y.
\]

A \emph{projective basis} of a correspondence is a finite sequence $(\xi_{i})^{n}_{i=1}\subseteq X$ such that
\begin{equation}\label{eq:projective-basis-cond}
x=\sum_{i} \xi_{i}\braket{\xi_{i} | x}
\end{equation}
holds for all $x\in X$.
Note that the existence of a projective basis only depends on $X$ as a right Hilbert $A$ module, and not on the structure of the left action.
It is not hard to show that the existence of a projective basis is equivalent to $X$ being finitely generated projective as an (algebraic) right $A$-module.

We can use projective bases to make the connection with algebraic quantum field theory (AQFT), generalizing the Doplicher--Haag--Roberts (DHR) theory approach to superselection sectors.
In AQFT, the DHR C$^*$-tensor category can be defined in terms of localized, transportable endomorphisms of the quasi-local algebra.
An endomorphism $\rho$ is localized in a small region $F$ if $\rho(a)=a$ for all $a$ localized in the complement of $F$.
Transportability is defined by the property that for any small region $F$, $\rho$ is unitarily conjugate to an endomorphism localized there.

To translate this into correspondences, we note that an endomorphism $\rho\colon A\rightarrow A$ yields a correspondence $_{\rho} A$, which is $A$ as a right Hilbert $A$ module, with left action given by $\rho$.
However, this abstract bimodule does not recover the endomorphism itself, and thus we cannot a-priori make sense of localized, transportable bimodules.
However, if we choose the 1-element projective basis $1\in {}_{\rho} A$, then we recover $\rho$ as the matrix coefficient of the Hilbert module $\rho(a)=\braket{ 1 | a\cdot 1 }$.
Furthermore, if we pick any unitary $u\in {}_{\rho}A$, we see that that $\Ad(u)\circ \rho=\braket{ u | a \cdot u}$.
Thus the correspondence picture allows us to collect all unitary conjugacy classes of an endomorphism into one object, and recover them by picking (unitary) projective bases.

The idea behind DHR bimodules is to focus on this later point, but keep it a bit more general by not requiring endomorphisms (encoded by the special type of projective basis), but more generally \emph{amplimorphisms} encoded by arbitrary projective bases.
An amplimorphism is a (not necessarily unital) homomorphism $\pi\colon A\rightarrow M_{n}(A)$.
This is required to capture the full level of generality of super-selection theory, since in the setting of abstract spin chains, the local algebras are finite dimensional, and in particular are not purely infinite.
The idea to apply this to abstract spin chains has its origins in the earlier work of~\cite{MR1463825}, but was more developed recently in~\cite{2304.00068}.

We now give a formal definition of DHR bimodules, following~\cite{2304.00068}.
We denote the category of \emph{dualizable} C$^*$-correspondences over $A$ by $\Corr^d(A)$.

\begin{definition}
Let $A_{\bullet}$ be an abstract spin chain and $X$ a dualizable correspondence over the quasi-local algebra $A$.
Given a subset $F\subseteq \bZ$, we say a vector $\xi \in X$ is \emph{localized in} $F$ if we have
\[
a\xi_{i}=\xi_{i}a \quad \text{for all }a \in A_{F^{c}}.
\]
The bimodule $X$ is a called a \emph{DHR bimodule} if there exists an $r\ge 0$ such that for any interval $I$ of length at least $r$, there is a projective basis $(\xi_{i})_i$ for $X$ localized in $I$.
The full C$^*$-tensor subcategory of $\Corr^d(A)$ consisting of DHR bimodules is denoted $\DHR(A_{\bullet})$.
\end{definition}

Under Condition~\ref{cond:Haag}, the tensor category $\DHR(A_{\bullet})$ is braided~\cite{2304.00068}*{Theorem 3.13}.
The braiding can be explicitly described as follows.
For two DHR bimodules $X$ and $Y$, pick projective bases $\{\xi_{i}\}\subseteq X$ and $\{\eta_{j}\}\subseteq Y$ localized in $I$ and $J$ respectively, where $I$ and $J$ are sufficiently large intervals (whose precise size depends on the structure constants for $X$ and $Y$ respectively) with $I<<J$ (again the precise requirements depend on the constants from weak algebraic Haag duality).
Then $\{\xi_{i}\boxtimes_{A} \eta_{j}\}$ forms a projective basis for $X\boxtimes_{A}Y$, and the braiding is given by
\[
\beta_{X,Y}(\xi_{i}\boxtimes_{A} \nu_{j}) = \nu_{j}\boxtimes_{A} \xi_{i}\in Y\boxtimes_{A} X.
\]
This indeed gives a well-defined intertwiner of correspondences, does not depend on the choice of bases, and satisfies the hexagon equation that characterizes a braiding on a monoidal category.

\subsection{Charge transporters}

In this section, we introduce an analogue of the notion of `charge transporter'~\cite{MR1721563} from algebraic quantum field theory.
We note that charge transporters are essentially formalizations of ``patch operators'' used in topological holography~\cites{inamura202321dsymmetrytopologicalorderlocalsymmetric, PhysRevB.107.155136}.
Fix a DHR bimodule $X$ over $A_\bullet$.
Suppose that $F_1$, $F_2$ are finite subsets of $\bZ$, such that we have bases $(\xi^k_i)_i$ localized in $F_k$ for $k = 1, 2$.
We then put
\[
t_{i j} = \braket{ \xi^1_i | \xi^2_j } \in A.
\]
We thus have
\begin{equation}\label{eq:charge-transp-base-change}
\sum_i \xi^1_i t_{i j} = \xi^2_j
\end{equation}
for all $j$.
We note $t_{i j}^* = \braket{ \xi_j^2 | \xi_i^1 }$, hence
\begin{equation}\label{eq:charge-transp-base-change-adj}
\sum_j \xi_j^2 t_{i j}^* = \xi_i^1
\end{equation}
also holds.
In the following we write $T = [t_{i j}]_{i, j} \in M_{n_1\times n_{2}}(A)$, where $n_{1}$ and $n_{2}$ are the sizes of the bases localized in $F_{1}$ and $F_{2}$ respectively.

\begin{proposition}\label{prop:interaction-T-and-P1}
Consider the matrix $P_{1} = [p^{(1)}_{i j}]_{i, j} \in M_{n_1}(A)$ given by $p^{(1)}_{i j} = \braket{ \xi^1_i | \xi^1_j }$.
We then have $T T^* = P_{1}$ and $P_{1} T = T$, that is,
\begin{align*}
\sum_j t_{i j} t_{k j}^* &= \braket{ \xi^1_i | \xi^1_k },&
t_{i j} &= \sum_k \braket{ \xi^1_i | \xi^1_k } t_{k j}.
\end{align*}
\end{proposition}

\begin{proof}
This follows from
\[
a \braket{ \xi | \eta } = \braket{ \xi a^* | \eta }
\]
and~\eqref{eq:projective-basis-cond}.
\end{proof}

A similar computation demonstrates the following.

\begin{proposition}\label{prop:t-star-t-and-right-support-of-t}
Consider the matrix $P_{2} \in M_{n_2}(A)$ given by $P_{2} = [p^{(2)}_{i j}]_{i, j} \in M_{n_1}(A)$, where $p^{(2)}_{i j} = \braket{ \xi^2_i | \xi^2_j }$.
Then we have $T^* T = P_2$ and $T P_2 = T$, that is, we have
\begin{align*}
\sum_j t_{j i}^* t_{j k} &= \braket{ \xi^2_i | \xi^2_k },&
t_{i k} &= \sum_j t_{i j} \braket{ \xi^2_j | \xi^2_k }\,.
\end{align*}
\end{proposition}

Let $X$ and $Y$ be DHR bimodules over $A_\bullet$.
Let us recall some facts about the braiding
\[
\beta_{X, Y} \colon X \boxtimes_A Y \to Y \boxtimes_A X.
\]

Choose finite subsets $F_1 < F_3 < F_2$ of $\bZ$, where $<$ refers to the order on $\bZ$ rather than containment, such that $X$ has bases $(\xi^1_i)_i$ and $(\xi^2_j)_j$ localized in $F_1$ and $F_2$ respectively, and $Y$ has a basis $(\eta_k)_k$ localized in $F_3$.
Then by the independence of basis choice in defining the braiding, we can express $\beta_{X, Y}$ and $\beta_{Y, X}$ as $A$-linear extensions of the maps defined on bases by
\begin{align*}
\beta_{X, Y}(\xi^1_i \boxtimes \eta_k) &= \eta_k \boxtimes \xi^1_i,&
\beta_{Y, X}(\eta_k \boxtimes \xi^2_j) &= \xi^2_j \boxtimes \eta_k.
\end{align*}

\begin{proposition}\label{prop:monodromy-formula}
The monodromy $\beta_{Y, X} \beta_{X, Y} \colon X \boxtimes_A Y \to X \boxtimes_A Y$ is given by
\begin{equation}\label{eq:monodromy-formula}
\beta_{Y, X} \beta_{X, Y}(\xi^1_i \boxtimes \eta_k) = \sum_{l, j} \xi^1_{l} \boxtimes t_{l j} \eta_k t_{i j}^*.
\end{equation}
\end{proposition}

\begin{proof}
Using~\eqref{eq:charge-transp-base-change-adj}, we write
\[
\beta_{X, Y}(\xi^1_i \boxtimes \eta_k) = \eta_k \boxtimes \xi^1_i = \sum_j \eta_k \boxtimes \xi^2_j t_{i j}^*.
\]
Then the effect of the $A$-module homomorphism $\beta_{Y, X}$ becomes
\[
\sum_j \xi^2_j \boxtimes \eta_k t_{i j}^*.
\]
Then using~\eqref{eq:charge-transp-base-change}, we get the right hand side of~\eqref{eq:monodromy-formula}.
\end{proof}

\section{Distinguished algebra objects}\label{sec:dist-alg-objs}

\subsection{Lagrangian algebra in the Drinfeld center}

Let $\cC$ be a unitary fusion category, and let us denote the set of its irreducible classes by $\Irr(\cC)$.
Given $i \in \Irr(\cC)$, we denote its representative by $U_i$.
We also write $0$ for the class of monoidal unit, so that we have $U_0 = 1_\cC$.

We denote the (unitary) \emph{Drinfeld center} of $\cC$ by $\cZ(\cC)$ (for example, see~\cite{MR1966525}).
Its objects are pairs $(Z, c)$, consisting of an object $Z\in\cC$ together with a half braiding, that is, a natural unitary isomorphism
\[
c_X \colon Z \otimes X \to X \otimes Z \qquad (X \in \cC).
\]
that fulfills the hexagon axiom.
Morphisms in $\cZ(\cC)$ are morphisms in $\cC$ between the underlying objects that commute with the respective half braidings.
The monoidal product of $\cC$ induces a monoidal structure on $\cZ(\cC)$ where the half braiding of a product $Z \otimes Z'$ is given by the composition of the respective half braidings.
Moreover, $\cZ(\cC)$ is naturally unitarily braided with braiding
\begin{equation}\label{eq:braiding_center}
\beta_{(Z, c), (Z', c')} = c_{Z'}\colon Z \otimes Z' \to Z' \otimes Z.
\end{equation}
for objects $(Z, c)$ and $(Z', c')$ in $\cZ(\cC)$.

Recall that there is a distinguished object $(Z^\reg, c^\reg)$ in $\cZ(\cC)$, given by
\begin{equation}\label{eq:reg-hlf-br-obj}
Z^\reg = \bigoplus_{i \in \Irr(\cC)} U_i^* \otimes U_i.
\end{equation}
Let us recall the unitary half-braiding $c^\reg$, cf.~\cite{MR3509018}*{Section 3.2}.
Given $X \in \cC$ and indexes $i$ and $j$ for $\Irr(\cC)$, we take a family of morphisms $v^{i j}_\alpha\colon U_i \otimes X \to U_j$ such that $\sum_{\alpha} v^{i j *}_\alpha v^{i j}_\alpha$ is the orthogonal projection onto the $j$-th isotypic summand of $U_i \otimes X$.
Then we define
\[
c_{X, i j} \colon U_i^* \otimes U_i \otimes X \to X \otimes X^* \otimes U_i^* \otimes U_i \otimes X \to X \otimes U_j^* \otimes U_j
\]
by setting
\[
c_{X, i j} = \sqrt{\frac{d_i}{d_j}} \sum_\alpha (\id_X \otimes v^{i j * \vee}_\alpha \otimes v^{i j}_\alpha) (\bar{R}_X \otimes \id_{U_i^* \otimes U_i \otimes X}),
\]
where $d_i$ is the categorical dimension of $U_i$, and $v^{i j * \vee}_\alpha \colon X^* \otimes U_i^* \to U_j^*$ is the morphism we obtain by composing $v^{i j *}_\alpha$ with the standard solutions of conjugate equations for $X$, $U_i$, and $U_j$.
If we collect $c_{X, i j}$ as a morphism $Z^\reg \otimes X \to X \otimes Z^\reg$, it is unitary, and satisfies the consistency conditions for half-braiding.

We can detect the triviality of the Müger center of $\cZ(\cC)$ by testing it against $(Z^\reg, c^\reg)$, as follows.

\begin{proposition}\label{prop:comm-with-reg-hlf-br-impl-triv}
Let $(Z, c)$ be an object of $\cZ(\cC)$ such that $\beta_{Z^\reg, Z} \beta_{Z, Z^\reg} = \id_{Z \otimes Z^\reg}$.
Then we have $Z \simeq 1_\cC^{\oplus k}$ for some integer $k$.
\end{proposition}

\begin{proof}
The morphism $\beta_{Z, Z^\reg}^{-1} = c_{Z^\reg}^{-1}$ respects the direct sum decomposition~\eqref{eq:reg-hlf-br-obj}.
In particular, for the direct summand $U_0 = 1_\cC$, we get the identity morphism of $Z$.
On the other hand, if we look at $\beta_{Z^\reg, Z} = c^\reg_Z$ out of the summand $U_0^* \otimes U_0$, we end up in the summands $Z \otimes U_j^* \otimes U_j$ for all $U_j$ contained in $Z$.
From $\beta_{Z^\reg, Z} = \beta_{Z, Z^\reg}^{-1}$, we get the claim.
\end{proof}

\begin{remark}\label{rem:hlf-br-on-triv-by-grading}
Thus, $(Z, c)$ belongs to the symmetric subcategory of $\cZ(\cC)$ formed by the objects isomorphic to $(1_\cC^{\oplus k}, c)$.
This category is monoidally equivalent to $\Rep \Gamma_\cC$, where $\Gamma_\cC$ is the universal grading group of $\cC$, see~\cite{MR4498161}*{Theorem 3.23}.
For example, any element $\phi$ of the character group $\hat{\Gamma}_\cC = \Hom(\Gamma_\cC, \bT)$ defines a half braiding $c^\phi$ on $1_\cC$ (up to the identification of $1_\cC \otimes X$ and $X \otimes 1_\cC$ with $X$) by
\[
c^\phi_X = \omega([X]) \id_X \colon X \to X
\]
for irreducible $X$, where $[X]$ is the element of $\Gamma$ represented by $X$.
\end{remark}

\subsection{Half-line algebras and Jones's basic extension}

We now introduce a condition on abstract spin chains that allows us to conclude finiteness of the DHR bimodule category.
First, consider the algebras
\begin{align*}
A_- &= A_{(-\infty, -1]},&
A_+ &= A_{[0, \infty)},&
B_0 &= A_- \otimes A_+,
\end{align*}
which are all C$^*$-subalgebras of the quasi-local algebra $A$.

Recall that a conditional expectation $E \colon C \to D$ for an inclusion of C$^*$-algebras $D \subset C$ is of \emph{finite index} if there is a positive constant $\lambda$ such that $E(a) \ge \lambda a$ holds for all positive $a \in C$.

\begin{condition}[Rationality]\label{cond:rationality}
We say that $A_{\bullet}$ is rational if there is a faithful conditional expectation $E\colon A \to B_0$ of finite index (in the sense of Watatani~\cite{MR996807}).    
\end{condition}

This allows us to consider the basic extension
\[
A \subset B_1 = \langle A, e \rangle,
\]
and the dual expectation $E_1 \colon B_1 \to A$.
In particular, $B_1$ admits an $A$-valued inner product induced by $E_1$.

\section{Half-line categories}\label{sec:half-line-category}

Since we can think of $B_1$ as $A \boxtimes_{B_0} A$ (see~\cite{MR996807}), given $A$-bimodules $X$ and $Y$, the space of $B_1$-module maps between $B_1 \boxtimes_A X$ and $B_1 \boxtimes_A Y$ (considered in the category of $A$-bimodules) is equivalent to the space of maps of $B_0$-$A$-modules between $X$ and $Y$.
We then look for a better description of the latter morphism system.

We assume that $A_\bullet$ satisfies weak algebraic Haag duality (Condition~\ref{cond:Haag}) and the covering property (Condition~\ref{cond:generation}).

\subsection{General setting}

\begin{proposition}\label{prop:X-minus-as-Hilb-bimod}
Let $X$ be a DHR bimodule over $A_\bullet$, and $(\xi_j)_j$ be a projective basis for $X$ localized in some finite subset of $(-\infty, -R-L-1]$, where $R$ and $L$ are the controlling constants respectively from Conditions~\ref{cond:Haag} and~\ref{cond:generation}.
Then the span $X_- = \bigvee_j \xi_j A_- \subset X$ is an $A_-$-$A_-$-correpondence, independent of the choice of $(\xi_i)_i$.
\end{proposition}

\begin{proof}
Let us first check that $X_-$ is a Hilbert $A_-$-module, that is, that  the inner product of vectors $\xi, \eta \in X_-$ belongs to $A_-$.

By $\braket{\xi a | \eta b} = a^* \braket{\xi | \eta} b$, we can reduce the claim to the case $\xi = \xi_i$ and $\eta = \xi_j$ for some indexes $i$ and $j$.
By the support assumption these vectors commute with $A_{[-R-1, \infty)}$, hence the same holds for their inner product $\braket{\xi_i | \xi_j}$.
We then obtain $\braket{\xi_i | \xi_j} \in A_{(-\infty,-1]}$ by Condition~\ref{cond:Haag}.

Let us next check that $X_-$ is closed under left multiplication by elements of $A_-$.
From~\eqref{eq:projective-basis-cond} we have
\[
a \xi_i = \sum_j \xi_j \braket{ \xi_j | a \xi_i },
\]
hence we need to check that $\braket{ \xi_j | a \xi_i }$ belongs to $A_-$ whenever $a$ is an element of $A_-$.

By Condition~\ref{cond:generation}, the algebra $A_-$ is generated by $A_{(-\infty, -R-1]}$ and $A_{[-R-L, -1]}$, hence it is enough to check the claim for these algebras.
If we have $a \in A_{[-R-L, -1]}$, then $a$ and $\xi_i$ commute by the support condition and we have $\braket{ \xi_j | a \xi_i } = \braket{ \xi_j | \xi_i } a$.
If we have $a \in A_{(-\infty, -R-1]}$, then $a \xi_i$ still commutes with $A_{[-R, \infty)}$, hence by the above argument we again have $\braket{ \xi_j | a \xi_i } \in A_-$.

Let $(\xi^1_i)_i$ and $(\xi^2_j)_j$ be two bases of $X$ as above, and let us check that they generate the same $A_-$-module.
Consider the charge transporters $t_{i j} = \braket{\xi^1_i | \xi^2_j}$.
By the support condition $t_{i j}$ commutes with $A_{[-R-1, \infty)}$, hence belongs to $A_-$.
This implies that the right $A_-$-module generated by the vectors $\xi^1_i$ contains the one generated by the $\xi^2_j$.
By symmetry we have the converse inclusion.
\end{proof}

\begin{definition}
We define $F_-$ to be the functor $\DHR(A_\bullet) \to \Corr(A_-)$ given by $X \mapsto X_-$.
\end{definition}

\begin{proposition}
The functor $F_- \colon \DHR(A_\bullet) \rightarrow \Corr(A_-)$ is endowed with the structure of a monoidal functor.
\end{proposition}

\begin{proof}
Let us fix DHR bimodules $X,Y\in \DHR(A_\bullet)$.
Take respective projective basis $(\xi_i)_i$ and $(\eta_j)_j$ localized in subsets of $(-\infty, -1]$, far enough from $-1$ as above.
On one hand, they become bases for $X_-$ and $Y_-$.
On the other the vectors $(\xi_i \otimes \eta_j)_{i, j}$ form a projective basis in $X \boxtimes_A Y$.

Thus, we have a well-defined isomorphism of Hilbert modules $(X \boxtimes_A Y)_{-} \to X_{-}\boxtimes_{A_{-}}Y_{-}$ characterized by
\begin{equation*}
(\xi_i \otimes \eta_j) a \mapsto \xi_i \otimes\eta_j a \quad (a \in A_-).
\end{equation*}
It is routine to check that this is a natural isomorphism of bifunctors required for the structure of monoidal functors.
\end{proof}

As a consequence, the objects $X_-$ form a C$^*$-tensor subcategory of $\Corr(A_-)$.
Let us take the idempotent completion of this category.

\begin{definition}
We define $\cC_-$ as the full C$^*$-tensor subcategory of $\Corr(A_-)$ whose objects are the subobjects of $A_-$-correspondences of the form $X_-$ for a DHR bimodule $X$.
\end{definition}

In the same way, we define $X_+ \in \Corr(A_+)$ for $X \in \DHR(A_\bullet)$, and define $\cC_+$ as the idempotent completion of the image of the induced tensor functor $F_+ \colon \DHR(A_\bullet) \to \Corr(A_+)$.
Under Condition~\ref{cond:B0-is-simple}, $\cC_\pm$ has a simple unit.
We tacitly assume this condition in the following.

\medskip
Let us record several auxiliary results that we will use later.

\begin{proposition}
Given $X \in \DHR(A_\bullet)$, we have a natural isomorphism of Hilbert $A$-modules
\[
X_- \boxtimes_{A_-} A \simeq X, \quad \xi \otimes a \mapsto \xi a.
\]
\end{proposition}

\begin{proof}
This follows from the fact that, if $(\xi_i)_i$ is a basis of $X$ localized in a negative interval far enough from $-1$ as above, then it is also an $A_-$-basis of $X_-$.
\end{proof}

\begin{proposition}\label{prop:A-min-hom-locality}
Let $R$, $L$ be as in Proposition~\ref{prop:X-minus-as-Hilb-bimod}, and let $T\colon X_- \to Y_-$ be a homomorphism of Hilbert $A_-$-bimodules.
Take bases $(\xi_i)_i \subset X$ and $(\eta_j)_j \subset Y$ that are localized in the interval $[-a, -R-L-1]$.
Then we have $\braket{ \eta_j | T(\xi_i) } \in A_{[-(a + R), -1]}$.
\end{proposition}

\begin{proof}
By assumption $\xi_i$ is central for $A_{(-\infty, -(a+1)]} \otimes A_{[0, \infty)}$, hence $T(\xi_i)$ has the same property, and the same holds for $\braket{ \eta_j | T(\xi_i) }$.
We also know that $\braket{ \eta_j | T(\xi_i) }$ belongs to $A_-$.
We then have
\[
\braket{ \eta_j | T(\xi_i) } \in A_{[-(a + R), R-1]} \cap A_-,
\]
hence the claim.
\end{proof}

\begin{remark}
If $(\xi_i)_i$ is localized in $[-a, -b]$, then the vector $\xi'_i = T(\xi_i)$ is central for the algebra $A_{(-\infty, -(a+1)]} \otimes A_{[-(b-1), -1]}$ since $T$ is an $A_-$-bimodule homomorphism.
Moreover, $\xi'_i$ is also central for $A_{[0, \infty)}$, since the target of $T$ is $Y_-$.
But we still don't have the commutation with $A_{[-c, d]}$ for $c, d > 0$, which prevents us from concluding $\braket{ \eta_j | T(\xi_i) } \in A_{[-(a+K), -(b-K)]}$.
(Of course, this would be true for $A$-bimodule homomorphisms.)
\end{remark}

\subsection{Subalgebra from change transporters}

Now, given the definition of charge transporters, we will consider subalgebras defined by them, and introduce a condition on our abstract spin chain which guarantees the charge transporters locally generate the quasi-local algebra.

\begin{definition}
We define $C_0 \subset A$ to be C$^*$-subalgebra generated by $B_0$ and the elements $\braket{ \xi | \eta }$ for $\xi \in X_-$ and $\eta \in X_+$, for all DHR bimodules $X$.
\end{definition}

\begin{proposition}\label{prop:reduced-algebra-generation}
The algebra $C_0$ is linearly spanned by $\braket{ X_- | X_+ }$ for $X\in \DHR(A_{\bullet})$.
\end{proposition}

\begin{proof}
We have $B_0 \braket{ X_- | X_+ } \subset \braket{ X_- | X_+ }$ by the $A_\pm$-bimodule property of $X_\pm$ and the consistency of inner product.
The self-adjointness of the linear span of $\braket{ X_- | X_+ }$ comes from the fact that $\DHR(A_\bullet)$ is closed under duality, so that we can write
\[
\braket{ \xi | \eta }^* = \braket{ \eta | \xi } = \braket{ \bar{\xi} | \bar{\eta} }
\]
using the vectors $\bar{\xi}, \bar{\eta} \in X^*$ corresponding to $\xi$ and $\eta$.

It remains to prove that, when $X$ and $Y$ are DHR bimodules, we can find bases $(\xi^1_i)_i$, $(\xi^2_j)_j$, $(\eta^1_k)_k$, and $(\eta^2_l)_l$ localized in some intervals $F^X_1, F^Y_1 \subset (-\infty, -1]$ and $F^X_2, F^Y_2 \subset [0, \infty)$ such that the charge transporters for some DHR bimodule $Z$ contains the products $\braket{ \eta^1_k | \eta^2_l } \braket{ \xi^1_i | \xi^2_j }$.

We claim that $Z = X \boxtimes_A Y$ will do.
Take $F^Y_1$ to the left of $F^X_1$, and $F^Y_2$ to the right of $F^X_2$.
Then $\braket{ \xi^1_i | \xi^2_j }$ will commute with both $\eta^1_k$ and $\eta^2_l$.
We thus have
\[
\braket{ \xi^1_i \boxtimes \eta^1_k | \xi^2_j \boxtimes \eta^2_l } = \braket{ \eta^1_k | \braket{\xi^1_i | \xi^2_j} \eta^2_l} = \braket{ \eta^1_k | \eta^2_l } \braket{ \xi^1_i | \xi^2_j },
\]
establishing the claim.
\end{proof}

Now, let us consider the following conditions.

\begin{condition}[Charge-transporter generation]\label{cond:reduced-net}
The algebra $C_0$ agrees with $A$.
\end{condition}

Instead of cutting at $0$, we can consider the cut at a different integer $x \in \bZ$.

\begin{condition}\label{cond:reduced-net-arbitrary-cut-point}
Let $C_x \subset A$ be the subalgebra linearly spanned by $\braket{ X_{(-\infty, x-1]} | X_{[x, \infty)} }$ for the DHR bimodules $X$.
Then $C_x$ agrees with $A$.
\end{condition}

Another variation, which we do not use, would be the following.

\begin{condition}
Let $I = [c, d] \subset \bZ$ be an interval.
Consider the subalgebra $C_{I, k}$ of $A_{[-k, k]}$ generated by $A_I$, $A_{[-k, -c-1]}$, $A_{[d+1, k]}$, the charge transporters of DHR bimodules $X$ with respect to bases $(\xi_i)_i$ and $(\xi'_j)_j$, localized in $I$ and $[d+1, k]$, and the charge transporters of $X$ with respect to bases $(\eta_i)_i$ and $(\eta'_j)_j$, localized in $[-k, -c-1]$ and $I$.
For long enough $I$ (we will take $\absv{I} > 3 L$) and big enough $k$ (we will take $k > \absv{I} + L$), the algebra $C_{I, k}$ agrees with $A_{[-k, k]}$.
\end{condition}

This corresponds to the assumption in~\cite{MR1721563}*{Corollary 4.3}.
While it seems a natural requirement, there might be some counterexample of abstract spin chains with this property in the line of Remark 3 following~\cite{MR1721563}*{Corollary 4.3}.
In particular, can we produce a nontrivial decomposition of the form~\eqref{eq:decomp-of-center-by-dhr} from this observation?

\begin{proposition}\label{prop:decomp-A-as-B0-bimod}
Suppose that $A_\bullet$ satisfies charge-transporter generation, Condition~\ref{cond:reduced-net}.
Then there is a surjective $B_0$-bimodule homomorphism of the form
\[
\bigoplus_i (X_i)_- \otimes (Y_i)_+ \to A
\]
for some DHR bimodules $(X_i)_i$ and $(Y_i)_i$.
\end{proposition}

\begin{proof}
Given $Y$, we have a $B_0$-bimodule map $(Y^*_-) \otimes Y_+ \to A$ given by the natural pairing
\[
 \eta_1^* \otimes \eta_2 \mapsto \braket{ \eta_1 | \eta_2 }.
\]
By assumption and Proposition~\ref{prop:reduced-algebra-generation}, there are DHR bimodules $(Y_i)_i$ such that the subspaces $\braket{ (Y_i)_- | (Y_i)_+ } \subset A$ linearly span $A$.
\end{proof}

\begin{corollary}\label{cor:decomp-A-as-B0-bimod}
Let $X$ be a DHR bimodule.
As a $B_0$-bimodule, $X$ decomposes as $\bigoplus_i X^-_i \otimes X^+_i$, where $X^{\pm}_i$ are irreducible $A_{\pm}$-bimodules that appear as submodules of $Y_\pm$ for some DHR bimodule $Y$.
\end{corollary}

\begin{proof}
Let $\cD = \cC_- \boxtimes \cC_+$. The claim for $X = A$ follows from the above proposition by the semisimplicity of $\cD$.
Suppose that $A \simeq \bigoplus_i Z^-_i \otimes Z^+_i$ is such a decomposition.
Given a general DHR bimodule $X$, we have a surjective $B_0$-homomorphism $X_- \boxtimes_{A_-} A \to X$ given by the product map $\xi \otimes a \mapsto \xi a$.
This gives a surjective map $\bigoplus_i (X_- \boxtimes_{A_-} Z^-_i) \otimes Z^+_i \to X$.
Let us fix $i$.
If $Z^-_i$ appears as an $A_-$-submodule of $Y_-$ for some DHR bimodule $Y$, then $X_- \boxtimes_{A_-} Z^-_i$ appears as a submodule of $(X \boxtimes_A Y)_-$.
This shows the claimed decomposition of $X$.
\end{proof}

\subsection{Finiteness of the DHR category}

Let us put $\cD = \cC_- \boxtimes \cC_+$.
We assume that Condition~\ref{cond:B0-is-simple} holds, so that $\cD$ has a simple unit.

Let us denote by $\bA$ the C$^*$-algebra object in $\cD$ associated to the inclusion $B_0 \subset A$ (see~\cites{MR3948170,palomares2023discrete}).
The finite index condition on $E$ allows us to write
\begin{equation*}
\bA(X) = \Hom_{B_0 \mhyph B_0}( X, A) \simeq \Hom_{B_0 \mhyph A}(X \underset{B_0}{\otimes} A,A).
\end{equation*}
Using the rightmost model for $\bA$, the algebra structure can be defined as
\begin{equation*}
\bA(X) \otimes \bA(Y) \ni f \odot g \mapsto f \circ(\id_X \otimes g) \in \bA(X \underset{B_0}{\otimes}Y).
\end{equation*}
The involution $j^{\bA}_X\colon \bA(X) \to \bA(\overline{X})$ is given as the composite
\[
\Hom_{B_0\mhyph A}(X \underset{B_0}{\otimes}A ,A) \to \Hom_{B_0\mhyph A}(A, \overline{X} \underset{B_0}{\otimes} A) \overset{*}{\to} \Hom_{B_0\mhyph A}(\overline{X} \underset{B_0}{\otimes}A,A),
\]
where the first map is the Frobenius reciprocity map $f \mapsto (\id_{\overline{X}} \otimes f) (R_X \otimes \id_A)$.

\begin{remark}
Without the finite index condition, $\Hom_{B_0 \mhyph B_0}(X,A)$ would not give the right object.
\end{remark}

Corollary~\ref{cor:decomp-A-as-B0-bimod} implies that, under Condition~\ref{cond:reduced-net}, the map 
\[
|\bA|_{alg} = \bigoplus_{X \in \Irr(\cD)} X \odot \bA(X) \to A
\]
given by evaluation is an injective homomorphism with dense image.
The summand corresponding to the monoidal unit $1_\cD \in \Irr(\cD)$ is $B_0 \otimes (B_0' \cap A)$, since $\bA(B_0)$ is the space of $B_0$-central elements in $A$.
Thus, the orthogonal projection onto this summand defines a faithful conditional expectation
\[
\bE\colon A \to B_0 \otimes (B_0' \cap A).
\]

To estimate the ``index'' of $\bE$, we freely use notation from~\cite{MR2085108}.
In particular, given C$^*$-algebras $C$, $D$, and a bi-Hilbertian C$^*$-$C$-$D$-module ${}_C X_D$, we consider its \emph{right index element} $\rInd(X) \in C''$ and \emph{right numerical index} $\rI(X) \in [0, \infty]$.
The element $\rInd(X)$ can be formally considered as $F(\id_X)$ for the completely positive map $F$ extending $\cK(X_D) \to C$ given by $\xi \eta^* \mapsto {}_C(\xi, \eta)$, hence we have the additivity $\rInd(X \oplus X') = \rInd(X) + \rInd(X')$.
The number $\rI(X)$ can be defined as $\norm{\rInd(X)}$~\cite{MR2085108}*{discussion after Definition 2.17}.

\begin{lemma}\label{lem:r-ind-X-d-X}
For $X \in \cD$, we have $\rInd(X) = d(X)$.
\end{lemma}

\begin{proof}
The existence of finite right basis for $X$, combined with $B_0$ being unital, implies that $\rInd(X)$ is an element of $B_0$ by~\cite{MR2085108}*{Corollary 2.25}.
The bimodule $X$ is of finite right numerical index, that is, $\rI(X) < \infty$ holds.
This, combined with the simplicity of $B_0$, implies that $\rInd(X)$ is a scalar by~\cite{MR2085108}*{Corollary 2.26}.
Moreover, solutions to the conjugate equations for $X$ are related to $\rInd(X)$ by $\overline{R}^* \overline{R} = \rInd(X)$, see~\cite{MR2085108}*{Theorem 4.4}.
(We also have $R^* R = \mathrm{l}\mhyph\mathrm{Ind}(X)$.)
We then get the claim by normalization condition $d(X) = \overline{R}^* \overline{R}$.
\end{proof}

Recall that $\bA(X)$ is a bimodule over $\bA(1_\cD)$.
We get a left $\bA(1_\cD)$-valued inner product on $\bA(X)$ by
\[
{}_{\bA(1_\cD)} (f_1, f_2) = f_1 f_2^*,
\]
and the right $\bA(1_\cD)$-valued inner product by
\[
\braket{ f_1 | f_2 }_{\bA(1_\cD)} = (R^* \otimes \id_A) (\id_{\overline{X}} \otimes f_1^* f_2) (R \otimes \id_A).
\]
In general, these inner products are not full, and their ranges are ideals of $\bA(1_\cD)$.

Let $p_X$ be the central projection of $\bA(1_\cD)$ such that $p \bA(1_\cD)$ agrees with the range of left $\bA(1_\cD)$-valued inner product on $\bA(X)$.

\begin{proposition}\label{prop:r-ind-A-X-ge-d-X-inv-p-X}
We have $\rInd(\bA(X)) \ge d(X)^{-1} p_X$ for $X \in \cD$.
\end{proposition}

\begin{proof}
According to~\cite{MR2085108}*{Proposition 2.27}, we have $\rInd(\bA(X)) \ge \lambda' p_X$ for positive scalar $\lambda'$ whenever $\lambda' \norm{\braket{ f | f }_{\bA(1_\cD)}} \le \norm{ {}_{\bA(1_\cD)} (f, f)}$ holds for all $f \in \bA(X)$.
Thus, it is enough to check that $\lambda' = d(X)^{-1}$ satisfies this inequality.

By $\norm{R} = d(X)^{1/2}$, we have
\[
d(X)^{-1} \norm{\braket{ f | f }_{\bA(1_\cD)}} \le  d(X)^{-1} \norm{R}^2 \norm{f^* f} = \norm{f^* f}.
\]
By the C$^*$-identity, the elements $f f^* \in \bA(1_\cD)$ and $f^* f \in \End_{B_0 \mhyph A}(X \otimes A)$ have the same norm.
This proves the claim.
\end{proof}

Now, following~\cite{MR2085108}, we define $\rInd({}_{B_0} A_{B_0})$ to be the right index element $\rInd({}_{B_0} Y_{B_0})$, where $Y$ is the direct sum of the Hilbert $B_0 \otimes \bA(1_\cD)$-bimodules $X \otimes \bA(X)$ for $X \in \Irr(\cD)$.

\begin{theorem}
Let $\supp(A)$ denote the subset of $\Irr \cD$ formed by the irreducible objects $X$ such that $\bA(X) \neq 0$, and let $\sigma(Z(\bA(1_\cD)))$ be the set of minimal central projections of $\bA(1_\cD)$.
Then there is a partition of $\supp(A)$ into subsets $S_p$ labeled by $p \in \sigma(Z(\bA(1_\cD)))$, such that
\[
\rInd({}_{B_0} A_{B_0}) \ge \sum_{p \in \sigma(Z(\bA(1_\cD)))} \absv{S_p} p
\]
holds.
\end{theorem}

\begin{proof}
We have $\rInd(A) = \sum_{X \in \supp(A)} \rInd(X \otimes \bA(X))$ and $\rInd(X \otimes \bA(X)) = \rInd(X) \otimes \rInd(\bA(X))$.
Then the claim follows from Lemma~\ref{lem:r-ind-X-d-X} and Proposition~\ref{prop:r-ind-A-X-ge-d-X-inv-p-X}, by picking $p'_X \in \sigma(Z(\bA(1_\cD)))$ satisfying $p'_X \le p_X$ for each $X \in \supp(A)$.
\end{proof}

If $D \subset C$ is an inclusion of C$^*$-algebras with a conditional expectation $E\colon C \to D$ of finite index $\lambda$, $C$ becomes a bi-Hilbertian $C$-$D$-module and $\norm{\rInd(C)} = \lambda^{\prime -1}$ holds for a (possibly different) constant $\lambda' > 0$ such that $E - \lambda' \id_C$ is completely positive~\cite{MR2085108}*{Proposition 2.12 and remark after Definition 2.17}.
Note also that if $E$ is of finite index, then $C$ becomes a bi-Hilbertian $D$-bimodule such that $\rInd({}_D C_D) = \rInd({}_C C_D)$.

\begin{corollary}\label{cor:finite-index-implies-fusion-cat}
The conditional expectation $\bE \colon A \to B_0 \otimes (B_0' \cap A)$ is of finite index if and only if both $\cC_-$ and $\cC_+$ are unitary fusion categories.
\end{corollary}

\begin{proof}
By the above estimate, $\bE$ is of finite index if and only if $\supp(A)$ is a finite set.
Recall that the irreducible objects of $\cD = \cC_- \boxtimes \cC_+$ can be parameterized as 
\[
\Irr(\cD) = \set{ Y \boxtimes Y' | Y \in \Irr(\cC_-), Y' \in \Irr(\cC_+) }.
\]

Given $Y \in \Irr(\cC_-)$, let us take a DHR bimodule $X^Y$ such that $Y$ appears as a direct summand of $X^Y_-$.
By Section~\ref{sec:duality}, we have $\overline{X^Y} \in \DHR(A_{\bullet})$ as well.
Then $X^Y_- \otimes \overline{X^Y}_+$ maps to $A$ by the $B_0$-bimodule map
\[
\xi \otimes \overline{\eta} \mapsto \braket{ \eta | \xi }.
\]
Moreover, if $\xi$ is nonzero, then we can find $\overline{\eta}$ such that this pairing is nonzero.
Indeed, by choosing a basis $(\eta_i)_i$ localized in a positive interval, the equality $\xi = \sum_i \eta_i \braket{ \eta_i | \xi }$ implies that $\braket{ \eta_i | \xi } \neq 0$ should happen for some $i$.
Thus, for each $Y$ we have some $Y' \in \Irr(\cC_+)$ such that $\bA(Y \boxtimes Y')$ is nonzero.
This shows that, if $\Irr(\cC_-)$ is an infinite set, then $\supp(A)$ is also infinite.

It remains to show that $\Irr(\cC_-)$ is infinite if and only if $\Irr(\cC_+)$ is infinite.
By symmetry it is enough to show that the finiteness of $\Irr(\cC_-)$ implies the same for $\Irr(\cC_+)$.
If $\Irr(\cC_-)$ is finite, then Proposition~\ref{prop:F_fully_faithful} implies that $\DHR(A_{\bullet})$ is a full subcategory of $\cZ(\cC_-)$, hence is a fusion category.
Then, $\Irr(\cC_+)$ has to be finite as well.
\end{proof}

\subsection{Equivalence between \texorpdfstring{$\mathcal{C}_{+}$}{C+} and \texorpdfstring{$\mathcal{C}_{-}$}{C-} under algebraic Haag duality}

Let us further assume that $A_\bullet$ satisfies the algebraic Haag duality.
We will just use the property $A_- = A \cap A_+'$.

\begin{proposition}\label{prop:F-min-is-A-plus-central}
Given $X \in \DHR(A_\bullet)$, the $A_-$-bimodule $X_-$ can be described as
\[
X \cap A_+' = \set{ \xi \in X | \forall a \in A_+,   a \xi = \xi a },
\]
that is, the subspace of $A_+$-central vectors in $X$.
\end{proposition}

\begin{proof}
The inclusion of $X_-$ in $X \cap A_+'$ is obvious by construction, so let us show the reverse inclusion.
Let us fix a basis $(\xi_i)$ localized in $F \subset (-\infty, -1]$, and $\eta$ be a vector in $X \cap A_+'$.
Then the inner product $\braket{ \xi_i | \eta }$ commutes with $A_+$, and by the algebraic Haag duality, belongs to $A_-$.
This shows that $\eta$ is in the span of $\xi_i A_-$, which shows the claim.
\end{proof}
Given an $A_-$-bimodule map $f\colon X_-\to Y_-$ and a basis $(\xi_j)_j$ of $X$ localized in $F \subset (-\infty, -1]$ define 
\begin{equation}
 \tilde{f}\colon X \to Y,\quad x \mapsto \sum_j f(\xi_j) \braket{\xi_j | x}\,. 
\end{equation}
\begin{proposition}\label{prop:Mor_B0}
Let $A_\bullet$ be an abstract spin system satisfying algebraic Haag duality.
Given $X, Y \in \DHR(A_\bullet)$, the assignment 
\begin{equation}\label{eq:C_morphism_spaces}
\Hom_{\cC}(X_-, Y_-) \to \Hom_{B_0 \mhyph A}(X, Y),\quad f \mapsto \tilde{f}
\end{equation}
is a natural isomorphism.
\end{proposition}

\begin{proof}
First notice that the definition of $\tilde{f}$ is independent of the choice of projective basis of $X_-$.
This follows from Proposition~\ref{prop:interaction-T-and-P1} relating the charge transporters that connect two different choices of such bases.
That $\tilde{f}$ is a right $A$-module map is an immediate consequence of the $A$-linearity of the inner-product of $X$.
Now, for $b\in A_+$ we have that
\begin{equation}
\begin{aligned}
\tilde{f}(b x)=\sum_j f(\xi_j) \braket{\xi_j | b x}
&= \sum_j f(\xi_j) \braket{b^* \xi_j | x} = \sum_j f(\xi_j) \braket{ \xi_j b^* | x}\\
&= \sum_j f(\xi_j) b \braket{\xi_j | x} = \sum_j b f(\xi_j) \braket{ \xi_j | x} = b \tilde{f}(x),
\end{aligned}
\end{equation}
where the second equality uses that $(\xi_i)$ is negatively localized, and the last equality holds due to Proposition~\ref{prop:F-min-is-A-plus-central} since $f(\xi_j)\in Y_-$.
For $b\in A_-$
\begin{equation}
\begin{aligned}
\tilde{f}(b x)&= \sum_j f(\xi_j) \braket{ b^* \xi_j | x }= \sum_{i,j} f(\xi_j) \braket{\xi_i \braket{\xi_i | b^* \xi_j} | x} \\
&= \sum_{i,j} f(\xi_j) \braket{ \xi_i | b^* \xi_j }^{*} \braket{ \xi_i | x }= \sum_{i,j} f(\xi_j) \braket{ \xi_j | b \xi_i } \braket{ \xi_i | x } \\
&= \sum_{i} f\biggl(\sum_j\xi_j\braket{ \xi_j | b \xi_i }\biggr) \braket{ \xi_i | x }= \sum_{i} f\left(b \xi_i\right) \braket{ \xi_i | x }\\
&= \sum_i b f(\xi_i) \braket{ \xi_i | x } = b \tilde{f}(x),
\end{aligned}
\end{equation}
where the second equality expresses $b^* \xi_j$ in terms of the projective basis, and the fifth and sixth equalities uses the fact that $f$ is an $A_-$-bimodule map.
It follows that $\tilde{f}$ is a $B_0$-module map, and thus the assignment~\eqref{eq:C_morphism_spaces} is well-defined.

Restriction to the negative component of a DHR bimodule is the inverse of~\eqref{eq:C_morphism_spaces}.
Given a $B_0$-$A$-bimodule map $h\colon X\to Y$ and $x\in X_-$, then for every $a\in A_+$ we have that $h(x) a = h(x a)=h(a x)=a h(x)$ by Proposition~\ref{prop:F-min-is-A-plus-central}.
It follows again from Proposition~\ref{prop:F-min-is-A-plus-central} that $h(x)\in Y_-$.
For $f\in \Hom_{\cC}(X_-, Y_-)$ and $h\in\Hom_{B_0 \mhyph A}(X, Y)$, it follows from the definition of a projective basis that $\tilde{f}(x')=f(x')$ for $x'\in X_-$ and that $\tilde{h}(x)=h(x)$ for any $x\in X$.
\end{proof}

As a byproduct, we see that the category of negative half-line bimodules and the one of positive half-line bimodules are equivalent.

\begin{corollary}\label{cor:MorC-C+}
Given $X \in \DHR(A_\bullet)$, let $X_+$ be the $A_+$-bimodule constructed in the same way as $X_-$ from a basis $(\eta_i)_i$ localized in some $F \subset [0, \infty)$.
Let $\cC_+$ be the idempotent completion of the category of $A_+$-bimodules $X_+$.
The assignment
\begin{equation*}
\Hom_{\cC_-}(X_-, Y_-) \to \Hom_{\cC_+}(X_+, Y_+),\quad f\mapsto \tilde{f}|_{X_+}
\end{equation*}
extends to a natural isomorphism between $\cC_-$ and $\cC_+$.
\end{corollary}

\begin{proof}
In analogy to Proposition~\ref{prop:Mor_B0}, the morphism space $\Hom_{\cC_+}(X_+, Y_+)$ is isomorphic to $\Hom_{B_0 \mhyph A}(X, Y)$.
Thus, composing the corresponding isomorphisms of $\cC_-$ and $\cC_+$ morphism spaces the assertion follows.
\end{proof}

\begin{proposition}
If the net $A_\bullet$ fulfills strict algebraic Haag duality, then the C$^*$-tensor categories $\cC_-$ and $\cC_+^\opo$ are tensor equivalent.
\end{proposition}

\begin{proof}
The functor given by the assignment
\begin{equation*}
\Omega\colon \cC_- \to \cC_+^\opo, \quad
X_-\mapsto X_+, \quad
f\mapsto \tilde{f} |_{X_+} \quad (X \in \DHR(A_\bullet))
\end{equation*}
is essentially surjective by definition, and fully faithful by Corollary~\ref{cor:MorC-C+}, and thus, an equivalence of categories.
The restriction of the DHR braiding $\beta\colon (X \boxtimes_A Y)_+ \to Y_+ \boxtimes_{A_+} X_+$ endows $\Omega$ with a monoidal structure.
Indeed, $\beta$ obeys the corresponding hexagon axiom and is natural for $A_-$-bimodule maps, that is, the diagram
\begin{equation*}
\begin{tikzcd}
(X\boxtimes_A Y)_+ \ar[d,swap,"\Omega(f\otimes g)"]\ar[r,"\beta"]&Y_+\boxtimes_{A_+}X_+\ar[d,"\Omega(g)\otimes \Omega(f)"]\\
(X'\boxtimes_A Y')_+\ar[r,"\beta",swap]&Y'_+\boxtimes_{A_+}X'_+
\end{tikzcd}
\end{equation*}
commutes for every $f\in\Hom_{\cC_-} (X_-, X'_-)$ and $g\in \Hom_{\cC_-} (Y_-, Y'_-)$.
In order to check this, choose bases $(\eta_l^X)_l$ and $(\eta_k^Y)_k$ of $X$ and $Y$, respectively localized in $F_1<F_2\subset [0,\infty)$.
On one hand, we have that
\[
\Omega(g)\otimes \Omega(f)\circ \beta (\eta_t^X\otimes\eta_s^Y)=\tilde{g}(\eta_s^Y)\otimes \tilde{f}(\eta_t^X)
\]
On the other hand, $\Omega(f\otimes g)$ is explicitly given by
\[
\Omega(f\otimes g)(\eta_t^X\otimes\eta_s^Y) = \sum_{i,j} f(\xi_i^X) \otimes g(\xi^Y_j) \braket{\xi_j^Y | \braket{\xi_i^X | \eta_t^X} \eta_s^Y}.
\]
Now, $\eta_s^Y$ is localized to the right of $\eta_t^X$ and thus
\[
\braket{\xi_j^Y | \braket{\xi_i^X | \eta_t^X} \eta_s^Y} = \braket{ \xi_j^Y | \eta_s^Y }\braket{ \xi_i^X | \eta_t^X },
\]
which in turn leads to
\[
\Omega(f\otimes g)(\eta_t^X\otimes\eta_s^Y)=\sum_{i} f(\xi_i^X)\otimes \tilde{g}(\eta_s^Y)\braket{ \xi_i^X | \eta_t^X }.
\]
Now, by writing $\Omega(f\otimes g)$ in terms of the basis $(\eta_k^X\otimes\eta_l^Y)_{l,k}$, we obtain that
\[ 
\beta\circ \Omega(f\otimes g)(\eta_t^X\otimes\eta_s^Y) =
\sum_{i}\sum_{l,k} \eta_l^Y \otimes \eta_k^X \braket{\eta_l^Y | \braket{\eta_k^X | f(\xi_i^X)} \tilde{g}(\eta_s^Y)} \braket{ \xi_i^X | \eta_t^X }.
\]
Since $\eta_l^Y$ is localized to the right of $\eta_k^X$, it commutes with the adjoint of $\braket{ \eta_k^X | f(\xi_i^X) }$ and therefore
\[
\braket{\eta_l^Y | \braket{\eta_k^X | f(\xi_i^X)} \tilde{g}(\eta_s^Y)} = \braket{ \eta_k^X | f(\xi_i^X) }\braket{ \eta_l^Y | \tilde{g}(\eta_s^Y) }.
\]
After summing over $k$ we have that
\[ 
\beta\circ \Omega(f\otimes g)(\eta_t^X\otimes\eta_s^Y) =
\sum_{i,j}\sum_{l} \eta_l^Y\otimes f(\xi_i^X) \braket{ \eta_l^Y | g(\xi_j^Y) } \braket{ \xi_j^Y | \eta_s^Y }\braket{ \xi_i^X | \eta_t^X }.
\]
Finally, since $\tilde{g}(\eta_s^Y)$
is positively localized, then
\[ 
\beta\circ \Omega(f\otimes g)(\eta_t^X\otimes\eta_s^Y)=
\sum_{i}\sum_{l} \eta_l^Y \braket{ \eta_l^Y | \tilde{g}(\eta_s^Y) } \otimes f(\xi_i^X) \braket{ \xi_i^X | \eta_t^X }=
\tilde{g}(\eta_s^Y)\otimes \tilde{f}(\eta_t^X),
\]
which proves the assertion.
\end{proof}

\section{Comparison with the Drinfeld center}\label{sec:realiz-center}

\subsection{Fully faithful embedding into the center}

Our next goal is to show that the category of DHR bimodules over a net satisfying Conditions~\ref{cond:B1-is-DHR} and~\ref{cond:reduced-net} can be realized as the Drinfeld center of an appropriate C$^*$-tensor category.
The ansatz again comes from the correspondence between the regular half braiding $Z^\reg$ and the basic construction $B_1$.
Recall that when $\cC$ is a fusion category, $Z^\reg$ is an algebra in $\cZ(\cC)$ such that the category of $Z^\reg$-modules in $\cZ(\cC)$ is equivalent to $\cC$.
The forgetful functor $\cZ(\cC) \to \cC$ corresponds to the free module functor $Z \mapsto Z^\reg \otimes Z$.

\begin{lemma}\label{lem:nat_restr_braiding}
Given $X,Y\in \DHR(A_\bullet)$, the restriction 
\[c_{Y_-} \colon X_- \boxtimes_{A_-} Y_- \to Y_- \boxtimes_{A_-} X_-\]
of the DHR braiding $\beta$ is natural in $Y$ for $A_-$-bimodule maps and natural in $X$ for $A$-bimodule homomorphisms.
\end{lemma}

\begin{proof}
Let $T\colon Y_- \to Y'_-$ be a homomorphism of $A_-$-bimodules.
Let $F_1 < F_2$ be regions in $(-\infty, -1]$, and pick bases $(\xi_i)_i \subset X$, $(\eta_j)_j \subset Y$ localized in $F_1$ and $F_2$, respectively.
Then, by Proposition~\ref{prop:A-min-hom-locality}, when $(\eta'_k)_k \subset Y'$ is a basis localized in $F_2$, the coefficients $\braket{ \eta'_k | T(\eta_j) }$ belong to $A_{[-b, -1]}$ for some $b$.
By arranging $F_1$ to be disjoint from $[-b, -1]$, we can follow the same argument as the naturality of braiding for $A$-bimodule homomorphisms~\cite{2304.00068}*{Theorem 3.13}.
\end{proof}

\begin{remark}
If we try to apply $A_-$-bimodule homomorphisms to $X$, the above argument breaks down in the following way: $T(\xi_i)$ would be supported in $[-a, -1]$ that contains $F_2$, and the argument of~\cite{2304.00068}*{Theorem 3.13} does not carry over.
\end{remark}

To simplify the notation, we will write $\cC$ for $\cC_-$ from now on.

\begin{proposition}
The functor $F_-$ induces a braided monoidal functor
\[
\tilde{F} \colon \DHR(A_\bullet) \to \cZ(\cC), \quad X\mapsto (X_-,c).
\]
\end{proposition}

\begin{proof}
For $X\in \DHR(A_\bullet)$, we assign as a half braiding the restriction $c_{Y_-} \colon X_- \boxtimes_{A_-} Y_- \to Y_- \boxtimes_{A_-} X_-$ of the braiding on $\DHR(A_\bullet)$.
By Lemma~\ref{lem:nat_restr_braiding} this is a well-defined assignment.
The monoidal structure of $F_-$ induces a monoidal structure on $\tilde{F}$.
Finally, to see that $\tilde{F}$ is braided notice that the braiding in $\cZ(\cC)$ is defined via the half braiding~\eqref{eq:braiding_center}.
But, by construction, the half braiding of an object $X_-$ in $\cZ(\cC)$ comes from the braiding in $\DHR(A_\bullet)$.
\end{proof}

\begin{proposition}\label{prop:F_fully_faithful}
Suppose that $A_\bullet$ satisfies charge-transporter generation (Condition~\ref{cond:reduced-net}).
Then the functor $\tilde{F}$ is fully faithful.
\end{proposition}

\begin{proof}
We know that $\DHR(A_\bullet)$ is a rigid C$^*$-tensor category by Corollary~\ref{cor:localized-left-basis-for-DHR}.
In particular, $\tilde{F}$ is faithful.
To show fullness, we can use the Frobenius reciprocity to reduce the claim to
\begin{equation*}
\Hom_{\DHR(A_\bullet)}(A, X) \simeq \Hom_{\cZ(\cC)}(A_-, X_-).
\end{equation*}

Let us consider a morphism in the right hand side.
It is given by an $A_-$-central vector $\xi$ in $X_-$ such that $\beta_{X, Y}(\xi \boxtimes_A \eta) = \eta \boxtimes_A \xi$ holds for all $Y \in \DHR(A_\bullet)$ and $\eta \in Y_-$.

By construction $\xi$ is also $A_+$-central, hence it is a $B_0$-central vector.
We then have an interval $[-L, L]$ around $0$ such that $\xi$ is localized on $[-L, L]$.

Let us fix $Y$, and a right projective basis $(\eta^1_i)_i$ localized in an interval to the left of $-L$, and another basis $(\eta^2_j)_j$ localized to the right of $L$.
Let $t^Y_{i j}$ be the corresponding charge transporters.

On one hand, by $\eta^1_i \in Y_-$, we have $\beta_{X, Y}(\xi \boxtimes_A \eta^1_i) = \eta^1_i \boxtimes_A \xi$ as remarked above.
On the other, by definition of the braiding, we also have $\beta_{X, Y}(\xi \boxtimes_A \eta^2_j) = \eta^2_j \boxtimes_A \xi$.
We thus have an equality of the form
\begin{equation*}
\sum_i \eta^1_i \otimes t^Y_{i j} \xi = \sum_i \eta^1_i \otimes \xi t^Y_{i j}
\end{equation*}
This, together with commutation of $\braket{ \eta^1_{i'} | \eta^1_i }$ and $\xi$, implies that we have
\[
\sum_i \braket{ [t^Y_{i j}, \xi] | [t^Y_{i j}, \xi] } = 0.
\]
This implies $[t^Y_{i j}, \xi] = 0$ for all $i$ and $j$.

Now, by Condition~\ref{cond:reduced-net}, $\xi$ commutes with a generating set of $A$, hence it is $A$-central.
Thus, $\xi$ represents a morphism in $\Hom_{\DHR(A_\bullet)}(A, X)$.
\end{proof}

\begin{remark}
If we further assume Condition~\ref{cond:B1-is-DHR} and that $\DHR(A_\bullet)$ has only finitely many irreducible classes, there is an alternative argument as follows.
In this case we know that $\DHR(A_\bullet)$ is a modular tensor functor from Theorem~\ref{thm:triviality}, and $\tilde{F}$ is a braided tensor functor.
Then the claim follows from~\cite{Davydov2013a}*{Corollary 3.26}.
\end{remark}

\subsection{Nondegeneracy of the DHR category}\label{sec:mod-dri-cent}

In this section, our goal will be to show that under some assumption on the inclusion $B_0 \subset A$, the category of DHR bimodules is non-degenerately braided.

We assume that the net $A_{\bullet}$ is rational (Condition~\ref{cond:rationality}), i.e., there is a conditional expectation of finite index $E\colon A \to B_0$.
This allows us to consider the basic extension
\[
A \subset B_1 = \langle A, e \rangle,
\]
and the dual expectation $E_1 \colon B_1 \to A$.
In particular, $B_1$ admits an $A$-valued inner product induced by $E_1$.

We introduce the following condition, which will be used to prove non-degeneracy.
We will later see that $B_1$ agrees with the algebra $Z^{\reg}$ of Section~\ref{sec:dist-alg-objs} under this condition.

\begin{condition}[Locally aligned]\label{cond:B1-is-DHR}
We say that a rational net $A_{\bullet}$ is \emph{locally aligned} if the algebra $B_1$ is a DHR bimodule over $A_\bullet$.
\end{condition}

Observe that we are stating Condition~\ref{cond:B1-is-DHR} as stronger than rationality. That is, whenever we say an abstract spin chain $A_{\bullet}$ is locally aligned, it will be implicitly assumed that it is rational.

Our goal is to prove the following, motivated by Proposition~\ref{prop:comm-with-reg-hlf-br-impl-triv}.

\begin{proposition}\label{prop:trivial-bimod-characterization}
Let $A_\bullet$ be a net satisfying Condition~\ref{cond:B1-is-DHR}.
Suppose that $X$ is a DHR bimodule such that $\beta_{B_1, X} \beta_{X, B_1} = \id$.
Then, $X$ contains a nonzero $B_0$-central vector, hence contains $B_0$ as a $B_0$-bimodule.
\end{proposition}

For the moment let us assume that $X$ satisfies the above assumption.
Choose $F_1 < F_3 < F_2$ such that there is a basis $(\eta_k)_{k=1}^m$ of $B_1$ localized in $F_3$, such that $e$ appears as one of the $\eta_k$, say $\eta_0$.

As before, choose bases $(\xi^1_i)_{i=1}^n$ and $(\xi^2_j)_{j=1}^n$ of $X$ localized in $F_1$ and $F_2$.
We may and do assume that the elements $\braket{ \xi^1_i | \xi^1_j } = p^{(1)}_{i j}$ and $\braket{ \xi^2_i | \xi^2_j } = p^{(2)}_{i j}$ belong to $B_0$, hence commute with $e$.

\begin{lemma}\label{lem:t-ij-commute-with-e}
Let $(t_{i j})_{i, j}$ be the associated charge transporters.
We then have $t_{i j} e = e t_{i j}$.
\end{lemma}

\begin{proof}
Consider the matrices
\begin{align*}
C &= [ t_{i' j} e - e t_{i' j}]_{i', j} \in M_{n}(B_1),&
D &= [ t_{i j}^*]_{j, i} \in M_n(A)
\end{align*}
and the row vector $V = [\xi^1_{i'}]_{i'}$.

We first claim $V C D = 0$, which means
\[
\sum_{i', j} \xi^1_{i'} \otimes [t_{i', j}, e] t_{i j}^* = 0
\]
for all $i$.
By Proposition~\ref{prop:monodromy-formula}, this can be written as
\[
\beta_{B_1, X} \beta_{X, B_1}(\xi^1_i \otimes e) - \sum_{i', j} \xi^1_{i'} \otimes e t_{i' j} t_{i j}^*.
\]
The first terms is equal to $\xi^1_i \otimes e$ by our assumption.
As for the second term, after summation over $j$, on the right of tensor symbol Proposition~\ref{prop:interaction-T-and-P1} gives $e p^{(1)}_{i' i}$.
Since $p = \braket{ \xi^1_{i'} | \xi^1_i }$ commutes with $e$, we have
\[
\xi^1_{i'} p_{i' i} \otimes e = \xi^1_i \otimes e,
\]
which indeed has the claimed cancellation.

Next, we claim that the matrix $D^* = [t_{i j}]_{i, j}$ satisfies $C D D^* = C$, which will then imply $V C = 0$.
Using Proposition~\ref{prop:t-star-t-and-right-support-of-t} and the commutation of $e$ and $p^{(2)}_{j k}$, we compute the components of $C D D^*$ as
\[
\sum_{i, j} [t_{i' j}, e] t_{i j}^* t_{i k} = \sum_j [t_{i' j}, e] p^{(2)}_{j k} = [t_{i' k}, e],
\]
which are indeed the components of $C$.

Now we are ready to get the claim.
Let us fix $i$ and $j$, and take any vector $\eta$ in the bimodule $B_1$.
The vanishing of $V C$ means
\[
\sum_{i'} \xi^1_{i'} \otimes [t_{i' j}, e] = 0.
\]
If we write out the inner product of the left hand side with $\xi_i \otimes \eta$, we get
\[
\sum_{i'} \braket{ \xi^1_{i} \otimes \eta | \xi^1_{i'} \otimes [t_{i' j}, e] } = \sum_{i'} \braket{ \eta | \braket{ \xi^1_{i} | \xi^1_{i'}} (t_{i' j} e - e t_{i' j}) } = \sum_{i'}\braket{ \eta | p^{(1)}_{i i'} (t_{i' j} e - e t_{i' j}) }.
\]
Again using the commutation of $P_1$ with $e$ and Proposition~\ref{prop:interaction-T-and-P1}, we obtain $\braket{ \eta | [t_{i j}, e] } = 0$.
Since this must be $0$ for arbitrary $\eta$, we get the vanishing of $t_{i j} e - e t_{i j}$.
\end{proof}

\begin{remark}
Morally speaking, if $X$ has trivial monodromy with $B_1$, the charge transporters of $X$ cannot interact with the observables sitting at the ``gap'' $0 \in \bZ$ separating the regions $F_1$ and $F_2$.
This should be a property characterizing the vacuum state.
\end{remark}

\begin{proof}[Proof of Proposition~\ref{prop:trivial-bimod-characterization}]
As before, let us choose bases $(\xi^1_i)_i$ and $(\xi^2_j)_j$ be projective bases localized in a negative region $F_1$ and a positive region $F_2$.
Let us take a $B_0$-bimodule decomposition $\bigoplus_{i \in I} X^-_i \otimes X^+_i \simeq X$ from Corollary~\ref{cor:decomp-A-as-B0-bimod}.

We have $e t_{ij} e = E(t_{ij}) e$ from the characterzing condition of $e$, while Lemma~\ref{lem:t-ij-commute-with-e} implies $e t_{ij} e = t_{ij} e^2 = t_{ij}e$.
Since $a \mapsto a e$ is an injective linear map from $A$ to $B_1$, we obtain $t_{i j} \in B_0$.

Now, consider the span $Y = \bigvee_i \xi^1_i B_0$.
On the one hand, it is a $B_0$-submodule of $X$, since $A_- \xi^1_i$ is contained in the span of $\xi^1_{i'} A_-$ for $1 \le i' \le n$.
On the other, it contains the vectors $\xi^2_i$ by $t_{i j} \in B_0$, hence we have the equality $Y = \bigvee_j \xi^2_j B_0$.

Let us take a subset $J \subset I$ satisfying $Y = \bigoplus_{i \in J} X^-_i \otimes X^+_i$.
On one hand, as the vectors $\xi^1_i$ are $A_+$-central, we have $X^+_i \simeq A_+$ for $i \in J$.
On the other, as the vectors $\xi^2_j$ are $A_-$-central, we also have $X^-_i \simeq A_-$ for $i \in J$.

This implies $Y \simeq B_0^k$ as $B_0$-bimodules, and we find $B_0$-central vectors.
\end{proof}

\begin{theorem}\label{thm:triviality}
Suppose that $A_\bullet$ satisfies charge-transporter generation (Condition~\ref{cond:reduced-net})  and that it is locally aligned (Condition~\ref{cond:B1-is-DHR}). Let $X$ be in the Müger center of $\DHR(A_\bullet)$.
Then $X$ is isomorphic to the trivial bimodule $A^k$.
\end{theorem}

\begin{proof}
By induction, it is enough to find a copy of $A$ as a direct summand of $X$ as a $A$-bimodule.
Proposition~\ref{prop:trivial-bimod-characterization} gives a nonzero $B_0$-central vector $\xi \in X$.
It is enough to show that $\xi$ is in fact $A$-central.

Let us take a basis $(\xi_i)_i$ of $X$ localized in an interval $[-a, -1]$.
Then we claim that $\xi$ is of the form $\sum_i \xi_i a_i$, with $a_i \in A_{[-a - R, R]}$ where $R$ is the controlling constant from weak algebraic Haag duality (Condition~\ref{cond:Haag}).
Indeed, the $B_0$-centrality of $\xi$ implies that $a_i = \braket{ \xi_i | \xi }$ should commute with $A_{(-\infty, -a]}$ and $A_{[0, \infty)}$.
Then Condition~\ref{cond:Haag} implies the claim.

Let us analyze the condition $\beta_{Y, X} = \beta_{X, Y}^{-1}$ when $Y$ is another DHR bimodule.
We choose two localized bases $(\eta^1_p)_p$ and $(\eta^2_q)_q$ of $Y$, respectively to the left of $-a-R$ and to the right of $R$.
Let $t_{p q} = \braket{ \eta^1_p | \eta^2_q }$ be the corresponding charge transporter, so that we have
\[
\eta^2_q = \sum_p \eta^1_p t_{p q}.
\]
By definition of the braiding we have
\begin{align*}
\beta_{Y, X}(\eta^1_p \otimes \xi_i) &= \xi_i \otimes \eta^1_p,&
\beta_{X, Y}^{-1}(\eta^2_q \otimes \xi_i) &= \xi_i \otimes \eta^2_q,
\end{align*}
and by the commutation of the coefficients $a_i$ with $\eta^1_p$ and $\eta^2_q$, we have the same formulas when we replace $\xi_i$ with $\xi$.

Then the equality $\beta_{Y, X} = \beta_{X, Y}^{-1}$ implies that we have
\begin{equation*}
\sum_p \eta^1_p \otimes t_{p q} \xi = \sum_p \eta^1_p \otimes \xi t_{p q}
\end{equation*}
for each $q$.
We then obtain the claim by an argument analogous to the last part of the proof for Proposition~\ref{prop:F_fully_faithful}.
\end{proof}

\subsection{Essential surjectivity of comparison functor}

Let us keep Conditions~\ref{cond:reduced-net} and~\ref{cond:B1-is-DHR}.
Having established that $\DHR(A_\bullet)$ is a nondegenerate full braided subcategory of $\cZ(\cC)$,~\cite{MR1990929}*{Theorem~4.2} gives an equivalence of braided categories
\begin{equation}\label{eq:decomp-of-center-by-dhr}
\cZ(\cC) \simeq \DHR(A_\bullet) \boxtimes \cZ',
\end{equation}
where $\cZ'$ is a modular tensor category (more specifically $\cZ'$ is the centralizer of the essential image of $\DHR(A_\bullet)$ under $\tilde{F}$ in $\cZ(\cC)$).

\begin{proposition}\label{prop:Z-prime-full-emb-to-C}
The tensor functor $\cZ' \to \cC$ obtained by the restriction of the forgetful functor $\cZ(\cC) \to \cC$ is fully faithful.
\end{proposition}

\begin{proof}
Since $\cZ'$ is a rigid tensor category, we only need to prove fullness.
By Frobenius reciprocity, we can reduce this to $\Hom_{\cZ'}(1, Y) = \Hom_{\cC}(1, Y)$ for objects $(Y, c) \in \cZ'$.

An element in $T \in \Hom_{\cC}(1, Y)$ is represented by a $A_-$-central vector $\eta \in Y$.
We need to show that, for any $X \in \DHR(A_\bullet)$ and any $\xi \in X_-$, the half-braiding $c_X \colon Y \boxtimes_{A_-} X_- \to X_- \boxtimes_{A_-} Y$ sends $\eta \otimes \xi$ to $\xi \otimes \eta$.

Fix a basis $(\xi_i)_i$ in $X$ localized to the left of the support of $\eta$, and write $\xi = \sum_i \xi_i a_i$.
We then have
\[
\xi \otimes \eta = \sum_i \xi_i \otimes \eta a_i
\]
by the $A_-$-centrality of $\eta$.
In view of the $A_-$-bimodularity of $c_X$, it is enough to check $c_X(\eta \otimes \xi_i) = \xi_i \otimes \eta$.
This follows from the commutation of $(Y, c)$ and $(X, \beta_{X, \bullet})$.
\end{proof}

Let us assume algebraic Haag duality now.
We are going to observe that $B_1$ becomes isomorphic in $\cZ(\cC)$ to  the algebra object $Z^{\reg}$ of Section~\ref{sec:dist-alg-objs}, from which we can conclude $\DHR(A_\bullet) \simeq \cZ(\cC)$.
Recall that, under this assumption, we have an equivalence $\Omega\colon \cC = \cC_- \to \cC_+^{\opo}$ characterized by the isomorphism of $B_0$-$A$-bimodules $Y \boxtimes_{A_-} A \simeq \Omega(Y) \boxtimes_{A_+} A$.

The first step is to see that the image of $B_1$ in $\cZ(\cC)$ is a direct summand of $Z^\reg$.
Recall that $Z^\reg$ is the image of $1_\cC$ under the functor $L \colon \cC \to \cZ(\cC)$ which is (both left and right) adjoint to the forgetful functor $U \colon \cZ(\cC) \to \cC$.

The $B_0$-$A$-bimodule corresponding to $1_\cC$ is $A$.
Moreover, up to the full embedding of $\cC$ into the category of $B_0$-$A$-modules, the composition of functors
\[
\begin{tikzcd}[column sep=small]
\DHR(A_\bullet) \arrow[r] & \cZ(\cC) \arrow[r, "U"] & \cC
\end{tikzcd}
\]
corresponds to the restriction of scalars from $A$-bimodules to $B_0$-$A$-modules.
Thus, the adjoint functor $\cC \to \DHR(A_\bullet)$ sends $1_\cC$ to $B_1 = A \boxtimes_{B_0} A$.

This shows that $B_1$ is the image of $Z^\reg$ for the functor $\cZ(\cC) \to \DHR(A_\bullet)$ adjoint to the inclusion of $\DHR(A_\bullet)$ to $\cZ(\cC)$, that is, $B_1$ is a subobject of $Z^\reg$ in $\cZ(\cC)$.
We want to show that they agree.

\begin{proposition}
Let $Y \in \cC$ and $Y' \in \cC_+$ be irreducible bimodules.
The $B_0$-bimodule $Y \otimes Y'$ appears as a direct summand of $A$ if and only if $\Omega(Y)$ is the dual object of $Y'$.
\end{proposition}

\begin{proof}
On one hand, having a nonzero $B_0$-bimodule homomorphism $Y \otimes Y' \to A$ is equivalent to having a nonzero $B_0$-$A$-module homomorphism $Y \otimes Y' \boxtimes_{B_0} A \to A$.
On the other, by the above characterization of $\Omega$, we have
\[
Y \otimes Y' \boxtimes_{B_0} A \simeq Y' \boxtimes_{A_+} \Omega(Y) \boxtimes_{A_+} A.
\]
Combining these, an embedding of $Y \otimes Y'$ into $A$ corresponds to a nonzero $A_+$-bimodule homomorphism from $Y' \boxtimes_{A_+} \Omega(Y)$ to $A_+$, which is a characterization of $\Omega(Y)$ being dual to $Y'$.
\end{proof}

\begin{corollary}
Let $Y$ be an irreducible object in $\cC$.
Then $Y \otimes \Omega(Y^*)$ appears as a direct summand of $A$.
\end{corollary}

\begin{proof}
By the above proposition, it is enough to check that there is some nonzero homomorphism from $Y \otimes Y'$ to $A$, for some $Y' \in \cC_+$.
This was indeed observed in the proof of Corollary~\ref{cor:finite-index-implies-fusion-cat}.
\end{proof}

We now have a decomposition
\[
A \simeq \bigoplus_{Y \in \Irr(\cC)} Y \otimes \Omega(Y^*),
\]
refining Proposition~\ref{prop:decomp-A-as-B0-bimod}.

\begin{proposition}
The image of $B_1$ in $\cZ(\cC)$ is isomorphic to $Z^\reg$.
\end{proposition}

\begin{proof}
We already observed that $B_1$ becomes a subobject of $Z^\reg$, hence it is enough to show that the underlying $\cC$-objects agree.
By the above decomposition of $A$, we have
\[
B_1 \simeq \bigoplus_{Y_1, Y_2 \in \Irr(\cC)} Y_1 \boxtimes_{A_-} Y_2 \otimes \Omega(Y_1^*) \boxtimes_{A_+} \Omega(Y_2^*).
\]
To get the underlying object $(B_1)_-$ in $\cC$, we take the $A_+$-central vectors, or equivalently, replace $\Omega(Y_1^*) \boxtimes_{A_+} \Omega(Y_2^*)$ by $\Hom_{\cC_+}(A_+, \Omega(Y_1^*) \boxtimes_{A_+} \Omega(Y_2^*))$.
Since the $Y_i$ are irreducible, this is nontrivial (then $1$-dimensional) if and only if $Y_2$ is isomorphic to $Y_1^*$.
We thus have
\[
(B_1)_- \simeq \bigoplus_{Y \in \Irr(\cC)} Y \boxtimes_{A_-} Y^*,
\]
which proves the claim.
\end{proof}

\begin{theorem}
Suppose that the net $A_{\bullet}$ satisfies charge-transporter generation (Condition~\ref{cond:reduced-net}), algebraic Haag duality (Condition~\ref{cond:Haag}), and that it is locally aligned (Condition~\ref{cond:B1-is-DHR}).
Then $\DHR(A_\bullet)$ is equivalent to $\cZ(\cC)$.
\end{theorem}

\begin{proof}
We need to check that the category $\cZ'$ is spanned by the monoidal unit of $\cZ(\cC)$.
By Proposition~\ref{prop:Z-prime-full-emb-to-C}, it is enough to check this at the level of underlying objects.
Since we know that $B_1$ is isomorphic to $Z^\reg$, the claim follows from Proposition~\ref{prop:comm-with-reg-hlf-br-impl-triv}.
\end{proof}

\subsection{Converse for local alignment}

In this subsection we assume the algebraic Haag duality condition, and consider the unitary tensor equivalence $\Omega\colon \cC_- \to \cC^{\opo}$ from above.

\begin{theorem}
Let $A_\bullet$ be a rational abstract spin chain satisfying the algebraic Haag duality, i.e., assume Conditions~\ref{cond:rationality} and~\ref{cond:Haag}.
Suppose moreover that the canonical functor $\DHR(A_\bullet) \to \cZ(\cC_-)$ is an equivalence.
Then the basic construction $B_1$ is a DHR bimodule.
\end{theorem}

\begin{proof}
Given an $A$-$A$-correspondence $H$ with right $A$-valued inner product $\langle \xi , \eta \rangle_A$, we restrict $H$ to a $B_0$-$B_0$-bimodule and equip it with the right $B_0$-valued inner product
\[
\langle \xi\ |\ \eta \rangle_{B_0} = E \left( \langle \xi\ |\ \eta \rangle_A \right).
\]
Since $B_0 \overset{E}{\subset} A$ is a finite index inclusion, $H$ is complete for this $B_0$-valued inner product, and it becomes a $B_0$-$B_0$ correspondence.
This construction provides a functor $\pi\colon \Corr(A) \to \Corr(B_0)$. Let $\mathcal{E}$ be the full sub-category of dualizable $A$-$A $-correspondences $H$ such that $\pi(H)$ is in the image of $\cC_- \boxtimes \cC_+$ inside $\Corr(B_0)$. Then since the 2-category of C$^*$-algebras and correspondences is Q-system complete~\cite{MR4419534}, $\mathcal{E}$ is equivalent to the category of $A$-$A$ bimodules internal to $\cC_- \boxtimes \cC_+$. Then $B_{1}=A\boxtimes_{B_{0}} A\in \mathcal{E}$.

On the other hand, since $A$ is a realization of the symmetric enveloping Q-system in $\cC_-\boxtimes\cC^\opo_{-}\simeq \cC_- \boxtimes \cC_+$, then $\mathcal{E}\simeq Z(\cC_{-})$. But since $\pi(\DHR(A_\bullet))\subset \cC_-\boxtimes \cC_{+}$, $\DHR(A_\bullet)\subseteq \mathcal{E}$. Finally since $\DHR(A_\bullet)\simeq Z(\cC_-)$, then $\DHR(A_\bullet)=\mathcal{E}$ so $B_{1}\in \DHR(A_\bullet)$. \end{proof}

\section{Abstract spin chains from C\texorpdfstring{${}^*$}{*}-2-categories}\label{sec:lattice-from-inv-bimod}

In this section, we will construct abstract spin chains in the spirit of the fusion spin chains studied in~\cite{2304.00068}.
However, the categories we consider will be more general, and the resulting net of algebras will not be translation invariant.
We will illustrate the various conditions we have introduced in the previous section for this class of abstract spin chains.

Let us fix a rigid C$^*$-2-category $\cC=(\cC_{i j})_{i, j\in S}$ with finitely many $0$-cells.
Thus, we have a finitely many collection of rigid C$^*$-categories $\cC_i$, and invertible $\cC_i$-$\cC_j$-bimodule categories $\cC_{ij}$, together with monoidal products $\cC_{i j} \times \cC_{j k} \to \cC_{i k}$ admitting standard associativity.
Given $X \in \cC_{i j}$, we denote the normalized categorical trace on $\End(X)$ by $\tr_X$.
Note that we have
\[
\tr_{X \otimes Y} = \tr_X \circ (\id \otimes \tr_Y) = \tr_Y \circ (\tr_X \otimes \id)
\]
on $\End(X \otimes Y)$ for $Y \in \cC_{j k}$.

We label the discrete lattice $\bZ$ with objects in $\cC$, so that we have a sequence $(i_n)_{n \in \bZ}$ of indexes in $S$.
For $a,b\in\bZ$ denote $\cC_{a,b}=\cC_{i_a,i_b}$.
We fix a sequence of composable $1$-cells $X_n\in \cC_{{n},{n+1}}$ fitting into the following diagram of $0$- and $1$-cells:
\begin{equation*}
\begin{tikzcd}
\cdots \arrow[r] & i_n \arrow[r, "X_n"] & i_{n + 1} \arrow[r] & \cdots
\end{tikzcd}
\end{equation*}

We then have a net $A_\bullet$ defined by
\begin{equation}\label{eq:net-from-2-category}
A_{[a, b]} = \End(X_a \otimes \cdots \otimes X_b).
\end{equation}
As a shorthand notation, we will write
\[
X_{[a,b]} = X_a \otimes X_{a+1} \otimes \cdots \otimes X_b,
\]
so that we have $A_{[a,b]} = \End(X_{[a, b]})$.

\begin{assumption}
We choose $(X_n)_n$ so that $A$ has the trivial center.
\end{assumption}

\subsection{Longo--Roberts type recognition for C\texorpdfstring{$^*$}{*}-algebraic models}
Let us recall a variation of the Longo--Roberts realization of C$^*$-tensor categories~\cite{MR1444286}, in the framework of $2$-categories and bimodules.

We fix a sequence of composable $1$-cells $(X_{n})_{n < 0}$ indexed by negative integers.
In this section we write
\[
A = \varinjlim_{a \to -\infty} \End(X_{[a, -1]}).
\]
We keep the standing assumption that $A$ has trivial center.
Let $j$ be a $0$-cell, and $Y$ be an object of $\cC_{i_0, j}$.
We then have an algebra
\[
A_Y = \varinjlim_{a \to -\infty} \End(X_{[a, -1]} \otimes Y),
\]
which naturally contains $A$.

The algebra $A_Y$ has an distinguished faithful state $\omega$, characterized by $\omega(a) = \tr_{X_{[a, -1]} \otimes Y}(a)$ for $a \in \End(X_a \otimes \cdots X_{-1} \otimes Y)$.

We will cosider the conditional expectation $E^Y \colon A_Y \to A$ given as the limit of $\id_{X_{[a, -1]}} \otimes \tr_Y$ as $a \to -\infty$.

\begin{lemma}\label{lem:rel-comm-fin-dim}
The relative commutant $A_Y \cap A'$ is finite-dimensional.
\end{lemma}

\begin{proof}
On one hand, we have the Pimsner--Popa type inequality
\[
(\id_{X_{[a, -1]}} \otimes \tr_Y)(a) \ge \frac{1}{d_Y^2} a
\]
for positive elements $a \in \End(X_{[a, -1]} \otimes Y)$.
This shows that the conditional expectation $E^Y$ satisfies the same inequality.

On the other hand, when $a$ is an element of $A_Y \cap A'$, its expectation $E^Y(a)$ must be in the center of $A$, hence is a scalar.
Then this scalar should agree with $\omega(a)$.

This shows that $\omega$ is a state on $A_Y \cap A'$ with the Pimsner--Popa type inequality.
A C$^*$-algebra admits such a state if and only if it is finite-dimensional, hence we obtain the claim.
\end{proof}

Consider the conditional expectation $E_n \colon A_Y \to \End(X_{[-n, -1]} \otimes Y)$ given as the limit of
\[
E^m_n = (\tr_{X_{[-m, -n-1]}} \otimes \id) \colon \End(X_{[-m, -1]} \otimes Y) \to \End(X_{[-n, -1]} \otimes Y),
\]
under $m \to \infty$.
As the extreme case, we write $E_0$ for the conditional expectation $A_Y \to \End(Y)$.

\begin{proposition}\label{prop:LR-recog}
Suppose that $A$ has a trivial center.
The natural homomorphism from $\End(Y)$ into $A_Y \cap A'$ is an isomorphism.
\end{proposition}

\begin{proof}
First let us assume that there is a subsequence of integers $n_k$ such that the $0$-cells $i_{-n_k}$ agree with $i_0$, hence $X_{[-n_k, -1]}$ are all in the tensor category $\cC_0$.

Let $a$ be an element of $A_Y \cap A'$, and put $a_k = E_{n_k}(a)$.
Then $a_k$ belongs to the relative commutant $\End(X_{[-{n_k}, -1]} \otimes Y) \cap \End(X_{[-{n_k}, -1]})'$.
This means that, for each irreducible class $V$ in $\cC_0$, and when we take an isometric inclusion $u \colon V \to X_{[-{n_k}, -1]}$, the element
\begin{equation}\label{eq:isotypic-component}
(u^* \otimes \id_Y) a_k (u \otimes \id_Y) \in \End(V \otimes Y)
\end{equation}
is independent of $u$.
Let $a_{k, 0}$ be such an element for the case of $V = 1_{\cC_0}$, the monoidal unit.

Then $a_{k, 0}$ is an element of $\End(Y)$.
We are going to check $a = \lim_k a_{k,0}$, which implies the assumption in this case.
Since $A_Y \cap A'$ is finite-dimensional by Lemma~\ref{lem:rel-comm-fin-dim}, it is enough to check the convergence in the $2$-norm for $\omega$, that is,
\[
\omega((a - a_{k,0})^* (a - a_{k,0})) \to 0 \quad (k \to \infty).
\]

We first claim that
\[
\omega(a^* b) = \tr_Y(a_{k,0}^* b)
\]
holds for any $b \in \End(Y)$.
Indeed, $a^* b$ is still an element of $A_Y \cap A'$, hence $E^Y(a^* b)$ is a scalar equal to the left hand side.
On the other hand, computing~\eqref{eq:isotypic-component} for $a^* b$ and $V = 1_{\cC_0}$ picks up the factor $\tr_Y(a_{k,0}^* b)$, hence the claim.

Using this claim for $b = a_{k, 0}$, and taking conjugate, we now have
\[
\begin{split}
\omega((a - a_{k,0})^* (a - a_{k,0})) &= \omega(a^* a) - \tr_Y(a_{k,0}^* a_{k,0}) - \tr_Y(a_{k,0}^* a_{k,0}) + \tr_Y(a_{k,0}^* a_{k,0})\\
&= \omega(a^* a) - \tr_Y(a_{k,0}^* a_{k,0}).
\end{split}
\]

Now, since the $a_k$ converges to $a$ in norm, we have a norm approximation
\[
a_k^* a_k \sim a^* a \sim E_{n_k}(a^* a) = (a^* a)_k
\]
for big $k$.
In particular we can replace $a_{k,0}^* a_{k,0}$ by $(a^* a)_{k, 0}$ by a small norm error for big $k$.
Since the scalar $\tr_Y((a^* a)_{k, 0})$ agrees with $\omega(a^* a)$, we have the desired convergence.

Finally, let us remove the assumption that $i_0$ happens infinitely many times amonth the $i_{-n}$.
By the finiteness of $0$-cells, we can first find some $i_{-m}$ such that $i_{-n_k} = i_{-m}$ holds for some sequence $(n_k)_k$.
Then, by what we proved for $X_{[-m+1, -1]} \otimes Y$, we have
\[
A_Y \cap A' = \End(X_{[-m+1, -1]} \otimes Y) \cap A'.
\]
Now, the categories $\cC_{i j}$ all appear as non-full subcategories of $\cC_{-m}$.
In particular, we can find a copy of the projection
\[
e = \frac{1}{d_{X_{[-m+1, -1]}}} R_{X_{[-m+1, -1]}} R_{X_{[-m+1, -1]}}^*
\]
inside $A$.
If $a$ is an element of $A_Y \cap A'$, the commutation of $a$ and $e$ implies that we have $a = (\tr_{X_{[-m+1, -1]}} \otimes \id)(a)$, hence we obtain $a \in \End(Y)$.
\end{proof}

\subsection{Algebraic Haag duality}

Let us come back to the net $A_\bullet$ from a bi-infinite sequence $(X_n)_{n \in \bZ}$ of $1$-cells.
We are goint to show that the net $A_\bullet$ has the algebraic Haag duality property.

\begin{proposition}
Given two integers $a < b$, the natural inclusion
\[
A_{[a, b]} \to (A_{(-\infty, a-1]} \vee A_{[b + 1, \infty)})' \cap A
\]
is surjective.
\end{proposition}

\begin{proof}
Conisder the conditional expectations
\[
E'_n \colon A \to A_{(-\infty, n]},
\]
given as the limit of expectations
\[
(\id \otimes \tr_{X_{[n, m]}}) \colon A_{(-\infty, m]} \to A_{(-\infty, n]} \quad (m > n).
\]
We then have
\[
E'_n({A_{(-\infty, a-1]}}' \cap A) = {A_{(-\infty, a-1]}}' \cap A_{(-\infty, n]},
\]
which is equal to $A_{[a, n-1]} = \End(X_{[a, n-1]})$ by Proposition~\ref{prop:LR-recog}.

Let us fix $a \in (A_{(-\infty, a-1]} \vee A_{[b + 1, \infty)})' \cap A$, and set $a_n = E'_n(a)$.
We then have $a = \lim_n a_n$ and $a_n \in A_{[a, n-1]} \cap {A_{[b + 1, n-1]}}'$.
This implies $a \in A_{[a, \infty)} \cap {A_{[b+1, \infty)}}'$, and again by Proposition~\ref{prop:LR-recog} we get $a \in A_{[a, b]}$.
\end{proof}

\subsection{Charge-transporter generation}

Now, let us further assume that each $\cC_i$ is a unitary fusion category, so that it has only finitely many irreducible classes.
Under this assumption, we can consider the following condition.

\begin{definition}
Let $X=\{X_n\}_{n\in \mathbb{Z}}$ be a sequence of composable $1$-cells in $\cC$.
We say that it has the \emph{strong generation} property if there exists $r>0$ such that for every interval $[a,b]$ of length larger than $r$ we have $Y \leq X_{[a,b]}$ for any simple $Y\in \cC_{a,b}$.
\end{definition} 

Note that, under strong generation, the C$^*$-algebras of the form $A_{[a, \infty)}$ and $A_{(-\infty, a]}$ have trivial center.
Indeed, given two minimal central projections $z$, $z'$ of $A_{[a, b]}$ (which correspond to simple objects in $\cC_{a, b}$), the generating property implies that there is an element $a \in A_{[a, b + r]}$ such that $z a z'$ is nontrivial.
This implies that there is no nontrivial ideal in $A_{[a, \infty)}$, see for example~\cite{MR1402012}*{Section III.4}.

We now produce a braided tensor functor $F\colon \cZ(\cC) \to \DHR(A_\bullet)$, and show the charge transporters for DHR bimodules in the image of $F$ generate the quasi-local algebra $A$.

The center $\cZ( \cC)$ is the braided C$^*$-category $\mathrm{\textbf{PseudoNat}}(\id_ \cC,\id_ \cC)$ of pseudo-natural equivalences of the identity.
An object $Z$ in $\cZ( \cC)$ amounts to a collection of invertible $1$-cells $Z_i\in \cC_i$, for $i \in S$, and invertible $2$-cells $c_Y\colon Z_i\otimes Y\to Y\otimes Z_j$ for a $1$-cell $Y\in\cC_{ij}$, obeying appropriate hexagon axioms.

Given such an object $Z=(\{Z_i\}_{i\in S}, c_Y)$ in $\cZ( \cC)$, we denote $Z_a= Z_{i_a}$ for $a\in\mathbb{Z}$.
For an interval $I=[a,b]$ consider the correspondence
\begin{equation}
F_{[a,b]}(Z) = \Hom_\cC(X_{[a,b]}, X_{[a,b]} \otimes Z_b)
\end{equation}
where $f\in A_{[a,b]}$ acts on $F_{[a,b]}(Z)$ by precomposition on the right and by postcomposition with $f\otimes\id_{Z_a}$ on the left.
Now, for $I=[a,b]\subset[c,d]=J$ we use the half-braidings of $Z$ to define an intertwiner
\[
\eta_{I,J}\colon F_{[a,b]}(Z)\to F_{[c,d]}(Z),\quad f\mapsto \id\otimes c_{X_{[c,a]}}\circ\id\otimes f\otimes \id.
\]
We obtain an $A$-correspondence by taking the inductive limit
\begin{equation}\label{eq:bicat_F(Z)}
F(Z) = \varinjlim_{a,b} F_{[a,b]}(Z).
\end{equation}
For an interval $I=[a,b]$, the canonical map into the limit is denoted by $\eta_I\colon F_{[a,b]}(Z)\to F(Z)$.

\begin{lemma}
For any $Z \in \cZ( \cC)$, the $A$-correspondence $F(Z)$ defined in~\eqref{eq:bicat_F(Z)} fulfills the locality condition and thus belongs to $\DHR(A_\bullet)$.
\end{lemma}

\begin{proof}
The proof follows mutatis mutandis from the argument in~\cite{2304.00068}*{Lemma 4.15}.
We spell out some details next.

Let $[a,b]$ be an interval of diameter greater or equal to $r$.
For $Y \in \Irr(\cC_{a,b})$, let $\{v^Y_k\colon Y \to X_{[a,b]} \otimes Z_b\}_k$ be mutually orthogonal partial isometries with $\sum_k v^Y_k\circ (v^Y_k)^*$ equal to the projection of $X_{[a,b]} \otimes Z_b$ onto the $Y$-isotypical component.
In particular, 
\[
\sum_{Y \in \Irr(\cC_{a,b})} \sum_k v^Y_k \circ(v^Y_k)^* = \id_{X_{[a,b]} \otimes Z_b}.
\]
For each $Y 
\in \Irr(\cC_{a,b})$ there exists a projection $p^Y_{a,b}\colon X_{[a,b]} \to Y$.
Define
\[
w^Y_k = v^Y_k \circ p^Y_{a,b}\colon X_{[a,b]} \to X_{[a,b]} \otimes Z_b.
\]
By construction,
\[
\sum_{Y \in \Irr(\cC_{a,b})} \sum_k w^Y_k \circ(w^Y_k)^* = \id_{X_{[a,b]} \otimes Z_b}.
\]
In other words, $\{w^Y_k\}_{Y,k}$ constitute a projective basis for $F_{a,b}(Z)$ over $\End(X_{[a,b]})$.

Let $I = [a,b] \subset [c,d] =J$.
A routine computation shows that the map $\eta_{I,J}\colon F_{[a,b]}(Z) \to F_{[c,d]}(Z)$ preserves the projective basis, that is, $\{\eta_{I,J}(w^Y_k)\}_{Y,k}$ is also a projective basis for $F_{[c,d]}(Z)$ over $\End(X_{[c,d]})$:
\[ 
\begin{aligned}
\sum_{Y \in \Irr(\cC_{a,b}), k} \eta_{I,J}(w^Y_k) \circ\eta_{I,J}(w^Y_k)^* &=
\sum_{Y \in \Irr(\cC_{a,b}), k} \id\otimes c_{X_{[c,a]}}\circ\id\otimes w^Y_k\circ
 (w^Y_k)^*\otimes\id\circ\id\otimes c_{X_{[c,a]}}^*\\
&=
\id\otimes c_{X_{[c,a]}}\circ\id_{X_{[a,b]} \otimes Z_b}\circ\id\otimes c_{X_{[c,a]}}^*= \id_{X_{[c,d]} \otimes Z_d}
\end{aligned}
\]
Therefore, in the inductive limit $F(Z)$, we obtain a projective basis over $A$.

As in the fusion case, one verifies that $\{\eta_I(w^Y_k)\}_{Y,k}$ can be localized in any interval of diameter at least $r$.
Indeed, consider intervals $I = [a,b]$ and $J=[c,d]\subset (-\infty,a)$.
For $f\in A_J$ we have, by naturality of the half-braiding and functoriality of $\otimes$, that
\[
\begin{aligned}
\iota_{J}(f)\rhd\eta_I(w^Y_k)&=\eta_{[c,b]}(f\otimes \id_{X_{[d,b]}}\circ 
\id_{X_{[a,b]}}\otimes c_{X_{[c,a]}}\circ \id_{X_{[c,a]}}\otimes w^Y_k)\\
&=\eta_{[c,b]}(\id_{X_{[a,b]}}\otimes c_{X_{[c,a]}}\circ f\otimes \id_{X_{[d,a]}}\otimes w^Y_k)\\
&=\eta_{[c,b]}(\id_{X_{[a,b]}}\otimes c_{X_{[c,a]}}\circ \id_{X_{[c,a]}}\otimes w^Y_k \circ f\otimes \id_{X_{[d,b]}})= \eta_I(w^Y_k)\lhd\iota_{J}(f)
\end{aligned}
\]
A similar argument shows locality in the case that $J=[c,d]\subset (b,\infty)$.
\end{proof}
We obtain a functor 
\[
\cZ( \cC) \rightarrow \DHR(A_\bullet), \quad Z \mapsto F(Z).
\]

\begin{proposition}\label{prop:pointcutgenerationformultifusionmodel}
The abstract spin chain $A_{\bullet}$ satisfies charge-transporter generation, Condition~\ref{cond:reduced-net}.
\end{proposition}

\begin{proof}
Let $[a,b]$ be an interval of diameter at least $r$ containing the origin. From Proposition~\ref{prop:pointcutgenerationformultifusionmodel} one can deduce that the space of charge transporters $\langle F(Z)_-|F(Z)_+ \rangle$, with $Z \in \cZ(\cC)$, are obtained as limits of linear combinations of morphism $X_{[a,b]} \to X_{[a,b]}$ of the form
\[ (\id_{X_{[a,-1]}} \otimes g) q^* v_1^* v_2 p (f \otimes \id_{X_{[0,b]}}), \]
where $f \in \End(X_{[a,-1]})$, $g \in \End(X_{[0,b]})$, $v_1,v_2 \in \Hom(Y,X_{[a,b]}\otimes Z_b)$ and $p,q \in \Hom(X_{[a,b]},Y)$ for some irreducible object $Y$. Thus, it suffices to show that an arbitrary morphism $X_{[a,b]}$ can be as linear combinations of morphisms of the form $q^* v_1^* v_2 p$.

By complete reducibility, it suffices to show that, for $[a,b]$ large enough, every $f\colon X_{[a,b]} \to X_{[a,b]}$ factoring as $f = \lambda q^*p$ for projections $p,q\colon X_{[a,b]} \to Y$, where $Y \in \Irr(\cC_{i_a,i_b})$, fits in a commutative diagram of the form
\begin{center}
\begin{tikzcd}
{X_{[a,b]}} \arrow[dd, "f"'] \arrow[rr, "p"] & & Y \arrow[ld, "\lambda"] \arrow[dd, "v_2"] \\
 & Y \arrow[ld, "q^*"'] & \\
{X_{[a,b]}} & & {X_{[a,b]}\otimes Z_b} \arrow[lu, "v_1^*"] \arrow[ll, "q^*u^*"]
\end{tikzcd}
\end{center}
for some $Z \in \cZ( \cC)$ having $Z_a$ in the $a$-component.
This can be achieved by taking any $Z$ such that $X_{[a,b]}$ has non-trivial $Y$-isotypical component, and one can take $v_2$ to be an element such that $v_2p \in \Hom(X_{[a,b]}, X_{[a,b]} \otimes Z_b)$ is part of a basis for $\Hom(X_{[a,b]}, X_{[a,b]} \otimes Z_b)$ over $\End(X_{[a,b]})$, and similarly for $v_1$ and $v_1q$, as in the proof of the previous Lemma.
\end{proof}

\subsection{Local alignment}

Finally, we prove that our model from fusion $2$-category is locally aligned (Condition~\ref{cond:B1-is-DHR}).

\begin{proposition}
As an $A$-bimodule, $B_1$ is isomorphic to $F(Z^\reg)$. In particular, the abstract spin chain $A_{\bullet}$ is locally aligned.
\end{proposition}

\begin{proof}
We first recall that $A$ can be decomposed as
\[
\bigoplus_{Y \in \Irr \cC_0} \varinjlim \Hom(X_{[-k, -1]} \otimes Y, X_{[-k, -1]}) \otimes \Hom(X_{[0, k]}, Y \otimes X_{[0, k]})
\]
as a $B_0$-bimodule, with comparison maps
\begin{equation}\label{eq:A-as-B-0-bimod-map}
\Hom(X_{[-k, -1]} \otimes Y, X_{[-k, -1]}) \otimes \Hom(X_{[0, k]}, Y \otimes X_{[0, k]}) \to \End(X_{[-k, k]})
\end{equation}
given by $f \otimes g \mapsto (f \otimes \id) (\id \otimes g)$.
The algebra structure of $A$ corresponds to the coend structure of $\bigoplus_Y Y^* \boxtimes Y \in \cC_0 \boxtimes \cC_0$.

Then, $A \otimes_{B_0} A$ can be written as
\begin{equation}\label{eq:B-1-decomp-1}
\bigoplus_{Y, W \in \Irr \cC_0} \varinjlim \Hom(X_{[-k, -1]} \otimes W \otimes Y, X_{[-k, -1]}) \otimes \Hom(X_{[0, k]}, W \otimes Y \otimes X_{[0, k]}).
\end{equation}
By the Frobenius reciprocity this is isomorphic to
\begin{equation*}
\bigoplus_{Y, W \in \Irr \cC_0} \varinjlim \Hom(X_{[-k, -1]} \otimes W, X_{[-k, -1]} \otimes Y^*) \otimes \Hom(X_{[0, k]}, W \otimes Y \otimes X_{[0, k]}),
\end{equation*}
which maps to $F(Z^\reg)$ by a $B_0$-bimodule isomorphism analogous to~\eqref{eq:A-as-B-0-bimod-map}.

It remains to check that this is an homomorphism of $A$-bimodules.
The coend description of the product of $A$ implies that the right $A$-module structures are compatible.

To compare the left $A$-module structures, we consider another functor $F' \colon \cZ(\cC) \to \DHR(A_\bullet)$ given by
\[
F'(Z) = \varinjlim \Hom_\cC(X_{[-k, -1]} \otimes Z^* \otimes X_{[0, k]}, X_{[-k, k]}).
\]
Then the half-braiding and the Frobenius reciprocity gives isomorphisms
\[
\Hom_\cC(X_{[-k, k]}, X_{[-k, -1]} \otimes Z \otimes X_{[0, k]}) \to \Hom_\cC(X_{[-k, -1]} \otimes Z^* \otimes X_{[0, k]}, X_{[-k, k]}),
\]
inducing a natural isomorphism $F \simeq F'$.

For $Z = Z^\reg$, the space~\eqref{eq:B-1-decomp-1} maps to $F'(Z)$ via
\[
\bigoplus_{Y, W \in \Irr \cC_0} \varinjlim \Hom(X_{[-k, -1]} \otimes W \otimes Y, X_{[-k, -1]}) \otimes \Hom(Y^{\prime *} \otimes X_{[0, k]}, Y \otimes X_{[0, k]}).
\]
The compatibility with left $A$-module structures follows from the compatibility of isomorphisms $B_1 \simeq F(Z^\reg)$, $B_1 \simeq F'(Z^\reg)$, and $F(Z^\reg) \simeq F'(Z^\reg)$ we have so far.

Let us take elements
\[
a \in \Hom(X_{[-k, -1]} \otimes W \otimes Y, X_{[-k, -1]}), \quad
b \in \Hom(X_{[0, k]}, W \otimes Y \otimes X_{[0, k]}).
\]
Under the maps $B_1 \to F(Z^\reg)$ and $B_1 \to F'(Z^\reg)$, the element of $B_1$ represented by $a \otimes b$ in the decomposition~\eqref{eq:B-1-decomp-1} goes to
\begin{equation}\label{eq:img-of-a-b-1}
(a \otimes \id) (\id \otimes \bar{R}_Y \otimes \id) (\id \otimes b) \in \Hom(X_{[-k, k]}, X_{[-k, -1]} \otimes Y^* \otimes Y \otimes X_{[0, k]}) \subset F(Z^\reg)
\end{equation}
and
\begin{equation}\label{eq:img-of-a-b-2}
(a \otimes \id) (\id \otimes R_{W}^* \otimes \id) (\id \otimes b) \in \Hom(X_{[-k, -1]} \otimes W \otimes Y^{\prime *} \otimes X_{[0, k]}, X_{[-k, k]}) \subset F'(Z^\reg).
\end{equation}
Expanding the isomorphism $F(Z^\reg) \simeq F'(Z^\reg)$, the element~\eqref{eq:img-of-a-b-1} goes to the element represented by Figure~\ref{fig:hlf-br-on-a-b}, where $\alpha$ and $\beta$ are labels for a choice of mutually orthogonal isometries $M \to Y \otimes X_{[0, k]}$ and $V \to X_{[0, k]} \otimes M^*$, where $M$ runs over the irreducible classes of $\cC_{0, k+1}$, and the superscript $\vee$ represents the transpose with respect to duality.
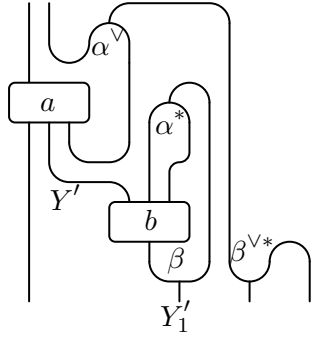
\begin{figure}[h]
\def \globalscale {1.000000}
\begin{tikzpicture}[y=1pt, x=1pt, yscale=\globalscale,xscale=\globalscale, every node/.append style={scale=\globalscale}, inner sep=0pt, outer sep=0pt]
  \path[fill=white] (9.0, 163.0) rectangle (182.0, -9.0);
  \path[draw=black,line cap=round,line join=round,line width=1.0pt] (20.0, 
  112.0) -- (20.0, 162.0);
  \path[draw=black,line cap=round,line join=round,line width=1.0pt] (55.0, 
  152.0) -- (55.0, 152.0).. controls (55.0, 157.523) and (59.477, 162.0) .. 
  (65.0, 162.0) -- (120.0, 162.0).. controls (125.523, 162.0) and (130.0, 
  157.523) .. (130.0, 152.0) -- (130.0, 142.0);
  \path[fill=white] (63.0, 52.0) -- (97.0, 52.0).. controls (98.657, 52.0) and 
  (100.0, 50.657) .. (100.0, 49.0) -- (100.0, 35.0).. controls (100.0, 33.343) 
  and (98.657, 32.0) .. (97.0, 32.0) -- (63.0, 32.0).. controls (61.343, 32.0) 
  and (60.0, 33.343) .. (60.0, 35.0) -- (60.0, 49.0).. controls (60.0, 50.657) 
  and (61.343, 52.0) .. (63.0, 52.0) -- cycle;
  \path[draw=black,line cap=round,line join=round,line width=1.0pt] (63.0, 52.0)
   -- (97.0, 52.0).. controls (98.657, 52.0) and (100.0, 50.657) .. (100.0, 
  49.0) -- (100.0, 35.0).. controls (100.0, 33.343) and (98.657, 32.0) .. (97.0,
   32.0) -- (63.0, 32.0).. controls (61.343, 32.0) and (60.0, 33.343) .. (60.0, 
  35.0) -- (60.0, 49.0).. controls (60.0, 50.657) and (61.343, 52.0) .. (63.0, 
  52.0) -- cycle;
  \path[fill=white] (13.0, 112.0) -- (47.0, 112.0).. controls (48.657, 112.0) 
  and (50.0, 110.657) .. (50.0, 109.0) -- (50.0, 95.0).. controls (50.0, 93.343)
   and (48.657, 92.0) .. (47.0, 92.0) -- (13.0, 92.0).. controls (11.343, 92.0) 
  and (10.0, 93.343) .. (10.0, 95.0) -- (10.0, 109.0).. controls (10.0, 110.657)
   and (11.343, 112.0) .. (13.0, 112.0) -- cycle;
  \path[draw=black,line cap=round,line join=round,line width=1.0pt] (13.0, 
  112.0) -- (47.0, 112.0).. controls (48.657, 112.0) and (50.0, 110.657) .. 
  (50.0, 109.0) -- (50.0, 95.0).. controls (50.0, 93.343) and (48.657, 92.0) .. 
  (47.0, 92.0) -- (13.0, 92.0).. controls (11.343, 92.0) and (10.0, 93.343) .. 
  (10.0, 95.0) -- (10.0, 109.0).. controls (10.0, 110.657) and (11.343, 112.0) 
  .. (13.0, 112.0) -- cycle;
  \path[draw=black,line cap=round,line join=round,line width=1.0pt] (40.0, 82.0)
   -- (40.0, 92.0);
  \path[draw=black,line cap=round,line join=round,line width=1.0pt] (40.0, 82.0)
   -- (40.0, 82.0).. controls (40.0, 76.477) and (44.477, 72.0) .. (50.0, 72.0) 
  -- (60.0, 72.0).. controls (65.523, 72.0) and (70.0, 76.477) .. (70.0, 82.0) 
  -- (70.0, 82.0);
  \path[draw=black,line cap=round,line join=round,line width=1.0pt] (30.0, 92.0)
   -- (30.0, 72.0).. controls (30.0, 66.477) and (34.477, 62.0) .. (40.0, 62.0) 
  -- (60.0, 62.0).. controls (65.523, 62.0) and (70.0, 57.523) .. (70.0, 52.0) 
  -- (70.0, 52.0);
  \path[draw=black,line cap=round,line join=round,line width=1.0pt] (20.0, 92.0)
   -- (20.0, -8.0);
  \path[draw=black,line cap=round,line join=round,line width=1.0pt] (120.0, 
  132.0) -- (120.0, 32.0).. controls (120.0, 26.477) and (115.523, 22.0) .. 
  (110.0, 22.0) -- (110.0, 22.0);
  \path[draw=black,line cap=round,line join=round,line width=1.0pt] (130.0, 
  60.817) -- (130.0, 39.5).. controls (130.0, 35.358) and (133.358, 32.0) .. 
  (137.5, 32.0) -- (137.5, 32.0).. controls (141.642, 32.0) and (145.0, 28.642) 
  .. (145.0, 24.5) -- (145.0, 17.0).. controls (145.0, 14.239) and (147.239, 
  12.0) .. (150.0, 12.0) -- (150.0, 12.0).. controls (152.761, 12.0) and (155.0,
   14.239) .. (155.0, 17.0) -- (155.0, 22.0).. controls (155.0, 27.523) and 
  (159.477, 32.0) .. (165.0, 32.0) -- (170.0, 32.0);
  \path[draw=black,line cap=round,line join=round,line width=1.0pt] (150.0, 
  12.0) -- (150.0, -8.0);
  \path[draw=black,line cap=round,line join=round,line width=1.0pt] (130.0, 
  142.0) -- (130.0, 52.0);
  \path[draw=black,line cap=round,line join=round,line width=1.0pt] (170.0, 
  32.0) -- (170.0, 32.0).. controls (175.523, 32.0) and (180.0, 27.523) .. 
  (180.0, 22.0).. controls (180.0, 22.0) and (180.0, 22.0) .. (180.0, 22.0) -- 
  (180.0, 22.0);
  \path[draw=black,line cap=round,line join=round,line width=1.0pt] (180.0, 
  22.0) -- (180.0, -8.0);
  \path[fill=white] (143.0, 17.0) -- (157.0, 17.0).. controls (158.657, 17.0) 
  and (160.0, 15.657) .. (160.0, 14.0) -- (160.0, 0.0).. controls (160.0, 
  -1.657) and (158.657, -3.0) .. (157.0, -3.0) -- (143.0, -3.0).. controls 
  (141.343, -3.0) and (140.0, -1.657) .. (140.0, 0.0) -- (140.0, 14.0).. 
  controls (140.0, 15.657) and (141.343, 17.0) .. (143.0, 17.0) -- cycle;
  \path[draw=black,line cap=round,line join=round,line width=1.0pt] (143.0, 
  17.0) -- (157.0, 17.0).. controls (158.657, 17.0) and (160.0, 15.657) .. 
  (160.0, 14.0) -- (160.0, 0.0).. controls (160.0, -1.657) and (158.657, -3.0) 
  .. (157.0, -3.0) -- (143.0, -3.0).. controls (141.343, -3.0) and (140.0, 
  -1.657) .. (140.0, 0.0) -- (140.0, 14.0).. controls (140.0, 15.657) and 
  (141.343, 17.0) .. (143.0, 17.0) -- cycle;
  \path[draw=black,line cap=round,line join=round,line width=1.0pt] (90.0, 22.0)
   -- (90.0, 22.0).. controls (92.761, 22.0) and (95.0, 19.761) .. (95.0, 17.0) 
  -- (95.0, 17.0).. controls (95.0, 14.239) and (97.239, 12.0) .. (100.0, 12.0) 
  -- (100.0, 12.0).. controls (102.761, 12.0) and (105.0, 14.239) .. (105.0, 
  17.0) -- (105.0, 17.0).. controls (105.0, 19.761) and (107.239, 22.0) .. 
  (110.0, 22.0).. controls (110.0, 22.0) and (110.0, 22.0) .. (110.0, 22.0) -- 
  (110.0, 22.0);
  \path[draw=black,line cap=round,line join=round,line width=1.0pt] (100.0, 
  12.0) -- (100.0, -8.0);
  \path[fill=white] (93.0, 17.0) -- (107.0, 17.0).. controls (108.657, 17.0) and
   (110.0, 15.657) .. (110.0, 14.0) -- (110.0, 0.0).. controls (110.0, -1.657) 
  and (108.657, -3.0) .. (107.0, -3.0) -- (93.0, -3.0).. controls (91.343, -3.0)
   and (90.0, -1.657) .. (90.0, 0.0) -- (90.0, 14.0).. controls (90.0, 15.657) 
  and (91.343, 17.0) .. (93.0, 17.0) -- cycle;
  \path[draw=black,line cap=round,line join=round,line width=1.0pt] (93.0, 17.0)
   -- (107.0, 17.0).. controls (108.657, 17.0) and (110.0, 15.657) .. (110.0, 
  14.0) -- (110.0, 0.0).. controls (110.0, -1.657) and (108.657, -3.0) .. 
  (107.0, -3.0) -- (93.0, -3.0).. controls (91.343, -3.0) and (90.0, -1.657) .. 
  (90.0, 0.0) -- (90.0, 14.0).. controls (90.0, 15.657) and (91.343, 17.0) .. 
  (93.0, 17.0) -- cycle;
  \path[draw=black,line cap=round,line join=round,line width=1.0pt] (85.0, 
  102.0) -- (85.0, 102.0).. controls (87.761, 102.0) and (90.0, 104.239) .. 
  (90.0, 107.0) -- (90.0, 107.0).. controls (90.0, 109.761) and (92.239, 112.0) 
  .. (95.0, 112.0) -- (95.0, 112.0).. controls (97.761, 112.0) and (100.0, 
  109.761) .. (100.0, 107.0) -- (100.0, 107.0).. controls (100.0, 104.239) and 
  (102.239, 102.0) .. (105.0, 102.0).. controls (105.0, 102.0) and (105.0, 
  102.0) .. (105.0, 102.0) -- (105.0, 102.0);
  \path[draw=black,line cap=round,line join=round,line width=1.0pt] (95.0, 
  132.0) -- (95.0, 112.0);
  \path[draw=black,line cap=round,line join=round,line width=1.0pt] (80.0, 32.0)
   -- (80.0, 32.0).. controls (80.0, 26.477) and (84.477, 22.0) .. (90.0, 22.0) 
  -- (90.0, 22.0);
  \path[fill=white] (88.0, 127.0) -- (102.0, 127.0).. controls (103.657, 127.0) 
  and (105.0, 125.657) .. (105.0, 124.0) -- (105.0, 110.0).. controls (105.0, 
  108.343) and (103.657, 107.0) .. (102.0, 107.0) -- (88.0, 107.0).. controls 
  (86.343, 107.0) and (85.0, 108.343) .. (85.0, 110.0) -- (85.0, 124.0).. 
  controls (85.0, 125.657) and (86.343, 127.0) .. (88.0, 127.0) -- cycle;
  \path[draw=black,line cap=round,line join=round,line width=1.0pt] (88.0, 
  127.0) -- (102.0, 127.0).. controls (103.657, 127.0) and (105.0, 125.657) .. 
  (105.0, 124.0) -- (105.0, 110.0).. controls (105.0, 108.343) and (103.657, 
  107.0) .. (102.0, 107.0) -- (88.0, 107.0).. controls (86.343, 107.0) and 
  (85.0, 108.343) .. (85.0, 110.0) -- (85.0, 124.0).. controls (85.0, 125.657) 
  and (86.343, 127.0) .. (88.0, 127.0) -- cycle;
  \path[draw=black,line cap=round,line join=round,line width=1.0pt] (40.0, 
  122.0) -- (45.0, 122.0).. controls (47.761, 122.0) and (50.0, 124.239) .. 
  (50.0, 127.0) -- (50.0, 127.0).. controls (50.0, 129.761) and (52.239, 132.0) 
  .. (55.0, 132.0) -- (55.0, 132.0).. controls (57.761, 132.0) and (60.0, 
  129.761) .. (60.0, 127.0) -- (60.0, 127.0).. controls (60.0, 124.239) and 
  (62.239, 122.0) .. (65.0, 122.0).. controls (65.0, 122.0) and (65.0, 122.0) ..
   (65.0, 122.0) -- (65.0, 122.0);
  \path[draw=black,line cap=round,line join=round,line width=1.0pt] (55.0, 
  152.0) -- (55.0, 132.0);
  \path[fill=white] (48.0, 147.0) -- (62.0, 147.0).. controls (63.657, 147.0) 
  and (65.0, 145.657) .. (65.0, 144.0) -- (65.0, 130.0).. controls (65.0, 
  128.343) and (63.657, 127.0) .. (62.0, 127.0) -- (48.0, 127.0).. controls 
  (46.343, 127.0) and (45.0, 128.343) .. (45.0, 130.0) -- (45.0, 144.0).. 
  controls (45.0, 145.657) and (46.343, 147.0) .. (48.0, 147.0) -- cycle;
  \path[draw=black,line cap=round,line join=round,line width=1.0pt] (48.0, 
  147.0) -- (62.0, 147.0).. controls (63.657, 147.0) and (65.0, 145.657) .. 
  (65.0, 144.0) -- (65.0, 130.0).. controls (65.0, 128.343) and (63.657, 127.0) 
  .. (62.0, 127.0) -- (48.0, 127.0).. controls (46.343, 127.0) and (45.0, 
  128.343) .. (45.0, 130.0) -- (45.0, 144.0).. controls (45.0, 145.657) and 
  (46.343, 147.0) .. (48.0, 147.0) -- cycle;
  \path[draw=black,line cap=round,line join=round,line width=1.0pt] (35.0, 
  162.0) -- (35.0, 127.0).. controls (35.0, 124.239) and (37.239, 122.0) .. 
  (40.0, 122.0) -- (40.0, 122.0);
  \path[draw=black,line cap=round,line join=round,line width=1.0pt] (70.0, 82.0)
   -- (70.0, 117.0).. controls (70.0, 119.761) and (67.761, 122.0) .. (65.0, 
  122.0) -- (65.0, 122.0);
  \path[draw=black,line cap=round,line join=round,line width=1.0pt] (80.0, 52.0)
   -- (80.0, 97.0).. controls (80.0, 99.761) and (82.239, 102.0) .. (85.0, 
  102.0) -- (85.0, 102.0);
  \path[draw=black,line cap=round,line join=round,line width=1.0pt] (90.0, 52.0)
   -- (90.0, 52.0).. controls (90.0, 57.523) and (94.477, 62.0) .. (100.0, 62.0)
   -- (100.0, 62.0).. controls (105.523, 62.0) and (110.0, 66.477) .. (110.0, 
  72.0) -- (110.0, 97.0).. controls (110.0, 99.761) and (107.761, 102.0) .. 
  (105.0, 102.0) -- (105.0, 102.0);
  \path[draw=black,line cap=round,line join=round,line width=1.0pt] (120.0, 
  132.0) -- (120.0, 132.0).. controls (120.0, 137.523) and (115.523, 142.0) .. 
  (110.0, 142.0) -- (105.0, 142.0).. controls (99.477, 142.0) and (95.0, 
  137.523) .. (95.0, 132.0) -- (95.0, 132.0);
\node[anchor=south west,line width=0.64pt] (text24) at (28.0, 98.0){$a$};
\node[anchor=south west,line width=0.64pt] (text25) at (78.0, 38.0){$b$};
\node[anchor=south west,line width=0.64pt] (text26) at (96.0, 2.0){$\beta$};
\node[anchor=south west,line width=0.64pt] (text27) at (142.0, 2.0){$\beta^{\vee *}$};
\node[anchor=south west,line width=0.64pt] (text28) at (90.0, 112.0){$\alpha^*$};
\node[anchor=south west,line width=0.64pt] (text29) at (50.0, 132.0){$\alpha^\vee$};
\end{tikzpicture}
\caption{Effect of $F(Z^\reg) \to F'(Z^\reg)$ on the image of $a \otimes b$}\label{fig:hlf-br-on-a-b}
\end{figure}

The part involving $b$, $\alpha^*$, and $\beta$ represets a morphism from $V$ to $W$.
Then the irreducibility of $W$ and $V$ forces $V = W$, and this part can be interpreted as the pairing of $(\alpha \otimes \id) (\id \otimes \beta^\vee) (R_M \otimes \id)$ and $b$.
Then a standard computation shows that the overall diagram is indeed equal to the element~\eqref{eq:img-of-a-b-2}.
\end{proof}

\subsection{Beyond the fusion case}

Once we remove the fusion condition on $\cC$, we can have models violating the charge transporter generation property, or local alignment.

\begin{example}
Suppose $\cC$ is the representation category of a profinite group $G = \varprojlim_n K_n$ for an inverse system of finite groups $(K_n)_{n \in \bN}$.
Then $\cC$ is a union of the fusion categories $\Rep K_n$.
With a suitable choice of objects $(X_n)_{n \in \bZ}$ generating $\cC$, we obtain an abstract spin chain $A_\bullet$ such that $\DHR(A_\bullet)$ contains the symmetric tensor category $\cC$.
\end{example}

Coming back to the setting of a rigid C$^*$-$2$-category $\cC$ with finitely many $0$-cells, let us fix a composable sequence of $1$-cells $(X_n)_{n \in \bZ}$, and take the abstract spin chain $A_\bullet$.
We assume that any object $Y$ is contained in $X_{[-n, -1]}^* \otimes X_{[-n, -1]}$ for big enough $n$.
This guarantees that
\[
E_Y = \varinjlim \Hom(X_{[-n, -1]}, X_{[-n, -1]} \otimes Y)
\]
is a nontrivial $A_-$-bimodule.

\begin{proposition}
Let $E$ be an object of $\DHR(A_\bullet)$.
Then the $A_-$-bimodule $E_-$ is isomorphic to $E_Y$ for some $Y \in \cC_0$.
\end{proposition}

\begin{proof}
Consider the $A_-$-bimodule $E^{(n)}$ defined as
\[
E^{(n)} = \varinjlim_m \End(X_{[-m, -1]}) \otimes_{\End(X_{[-m, -n]})} \End(X_{[-m, -1]}).
\]
We claim that, for big enough $n$, there is a surjective $A_-$-bimodule map from a direct sum of finitely many copies of $E^{(n)}$ to $E$.

Indeed, let $(\xi_i)_{i = 1}^d$ be a basis of $E$ localized on an interval of the form $[-n, -1]$.
Then, for each $i$, we have a bimodule map from $E^{(n)}$ to $E_-$ given by $a \otimes b \mapsto a \xi_i b$.

Next, by a usual Frobenius reciprocity argument, the bimodule $E^{(n)}$ can be identified with
\[
A_{X_{[-n+1,-1]}^*} = \varinjlim_m \End(X_{[-m, -1]} \otimes X_{[-n+1,-1]}^*) \cong E_{X_{[-n+1,-1]}^* \otimes X_{[-n+1,-1]}}.
\]
Consequently the $A_-$-endomorphisms of $E^{(n)}$ can be identified with the elements of $A' \cap A_{X_{[-n+1,-1]}^*}$.
By Proposition~\ref{prop:LR-recog}, the latter is the natural image of $\End(X_{[-n+1,-1]}^*)$.

This means that $E_-$ is a direct summand of $E_{X_{[-n+1,-1]}^* \otimes X_{[-n+1,-1]}}$ with respect to a projection in $\End(X_{[-n+1,-1]}^*)$ naturally acting on $X_{[-n+1,-1]}^* \otimes X_{[-n+1,-1]}$, hence we obtain the claim.
\end{proof}

Given $W \in \cC_{i j}$, let us write $\supp W$ for the subset of $\Irr \cC_{i j}$ consisting of the classes $V$ such that $V$ is a subobject of $W$.

\begin{proposition}
Given $E \in \DHR(A_\bullet)$, suppose $Y \in \cC_0$ is an object satisfying the above condition.
We then have $\supp X_{[-m, -1]}  = \supp X_{[-m, -1]} \otimes Y$ for big enough $m$.
\end{proposition}

\begin{proof}
Suppose $m$ is big enough so that we have a localized basis $(\xi_i)_i$ of $E$ on $[-m, -1]$.
We claim that $\supp X_{[-m, -1]}  = \supp X_{[-m, -1]} \otimes Y$ should hold for such $m$.

Suppose on the contrary that there is some class $V$ in $\supp X_{[-m, -1]} \otimes Y$ not contained in $\supp X_{[-m, -1]}$, and let $z$ be the central projection of $\End(X_{[-m, -1]} \otimes Y)$ corresponding to $V$.

Take a big enough $k$ such that there is a class $W$ in $\supp X_{[-k, -m-1]}$ such that $W \otimes V$ has nontrivial morphism to $X_{[-k, -1]}$.
Then, if $w$ is the central projection of $\End(X_{[-k, -m-1]})$ corresponding to $W$, the projection $w \otimes z$ acts nontrivially on $\Hom(X_{[-k, -1]}, X_{[-k, -1]} \otimes Y)$.

Since $(\xi_i)_i$ form a projective basis of $\Hom(X_{[-k, -1]}, X_{[-k, -1]} \otimes Y)$ over $\End(X_{[-k, -1]})$, the vectors of the form $\xi_i a$ for $1 \le i \le d$ and $a \in \End(X_{[-k, -1]})$ are total in $\Hom(X_{[-k, -1]}, X_{[-k, -1]} \otimes Y)$, and $w \otimes p$ acts nontrivially on the span of $\xi_i a \xi_j^*$ for $a \in \End(X_{[-k, -1]})$.
However, applying the partial trace we reduce this to the action of $p$ on the span of $\xi_i a \xi_j^*$ for $a \in \End(X_{[-m, -1]})$, which should be trivial.
The faithfulness of partial trace contradicts this.
\end{proof}

The above criterion can pose a strong control on the DHR bimodules when the fusion rules of $\cC$ are `torsion free' in a suitable sense.

For example, suppose that $\cC$ is a rigid C$^*$-tensor category with the fusion rules of a simple compact Lie group $G$.
Then the only possibility for the above $Y$ is a multiple of the trivial object.
Thus, as a left module, any object $E \in \DHR(A_\bullet)$ is a free $A$-module.
In view of Remark~\ref{rem:hlf-br-on-triv-by-grading}, as a monoidal category $\DHR(A_\bullet)$ is equivalent to the representation category of the commutative group $Z(G)$.
In this case the charge transporters from positive region to negative region belong to the subalgebra $B_0$, and we have $B_0 = C$, hence failure of the charge transporter generation.

\appendix

\section{Duality}\label{sec:duality}

Let $A_{\bullet}$ be an abstract spin chain. We next want to address the question of dualizability of DHR bimodules. 

\begin{condition}\label{cond:dualizability-as-corr}
Each DHR bimodule $X$ over $A_\bullet$ is dualizable as a correspondence over $A$.
\end{condition}

Recall that an $A$-$A$-correspondence $X$ is said to be \emph{a bi-Hilbertian $A$-$A$ C$^*$-bimodule}~\cite{MR2085108} if there is a left Hilbert $A$-module structure on $X$, and the topologies induced on $X$ by the left and right $A$-valued inner products are equivalent.

\begin{proposition}
Suppose Condition~\ref{cond:simplicity}.
If each DHR bimodule $X$ has a bi-Hilbertian $A$-$A$ C$^*$-bimodule structure, we have Condition~\ref{cond:dualizability-as-corr}.
\end{proposition}

\begin{proof}
Recall the notion of $X$ having a \emph{finite left numerical index}~\cite{MR2085108}: there is some constant $\lambda > 0$ such that
\[
\biggl\| \sum_i \braket{ \xi_i | \xi_i } \biggr\| \le \lambda \biggl\| \sum_i \xi_i^* \xi_i \biggr\|
\]
holds for any finite sequence $(\xi_i)_i \subset X$, where $\xi^* \eta$ denotes the operator $\zeta \mapsto \prescript{}{A}(\zeta, \xi) \eta$ on $\prescript{}{A} X$.
By~\cite{MR2085108}*{Theorem 4.8} we obtain this property from knowing that $X_A$ has a projective basis (hence $X$ is of finite \emph{right} numerical index).
Now,~\cite{MR2085108}*{Corollary 2.26} says that $X$ is of \emph{finite left index}, meaning that the series $\sum_j \braket{ \eta_j | \eta_j }$ converges in $\cM(A) = A$ for the strict topology, whenever $(\eta_j)_j$ is a generalized left basis of $X$.
By~\cite{MR2085108}*{Corollary 2.25}, we obtain a projective basis for $\prescript{}{A} X$.
\end{proof}

Now, for each finite subset $S \subset \bZ$, consider the projection
\[
e_S = \sum_i \xi_i^* \xi_i \in \cK(\prescript{}{A} X),
\]
where $(\xi_i)_i$ is any choice of projective basis for $\prescript{}{A_S} X_S$.
(In other words, it is the image of $1_{\cK(\prescript{}{A_S} X_S)} \in \cK(\prescript{}{A_S} X_S)$ under the not-necessarily unital embedding into $\cK(\prescript{}{A} X)$.)

\begin{proposition}\label{prop:big-S-has-left-basis}
Under the above assumptions, we have $e_S = 1$ for big enough $S$.
\end{proposition}

\begin{proof}
Without knowing that $\cK(\prescript{}{A} X)$ is unital, we already know that the family $(e_S)_S$ converges to $1 \in \cM(\cK(\prescript{}{A} X)) = \cL(\prescript{}{A} X)$ for the strict topology.
By the existence of finite projective basis, we actually have $\cK(\prescript{}{A} X) = \cL(\prescript{}{A} X)$, and we get $\norm{1 - e_S} \to 0$.
Since $e_S$ are projections, we must have equality $\norm{1 - e_S} = 0$ for big enough $S$.
\end{proof}

This shows that $X$ has a left basis localized in some finite set $S$, but we still do not have a control on the size and location of $S$.
Our next goal is to incorporate the quantitative versions of charge-transporter generation conditions to achieve this.

\begin{proposition}\label{prop:use-braid-switch-from-left-mod-to-right-mod}
Let $L$ be an integer so that any DHR bimodule have localized basis with support length $L$.
Given $M > 2L$, $M - L > a > L$, and a DHR bimodule $Y$, consider the subspaces
\begin{align*}
\prescript{Y}{}X_M &= \Span \set{ \braket{ \eta^+ | \eta^- } \xi^- | \xi^- \in X_{[-a-L,-a-1]}, \eta^- \in Y_{[-M,-1]}, \eta^+ \in Y_{[0,M]} }, \\
X^{Y}_M &= \Span \set{ \xi^- \braket{ \eta^+ | \eta^- } | \xi^- \in X_{[-a-L,-a-1]}, \eta^- \in Y_{[-M,-1]}, \eta^+ \in Y_{[0,M]} }.
\end{align*}
Then there is a right $A_{[-M,-1] \cup [0, M]}$-linear isomorphism from $X^{Y}_M$ to $\prescript{Y}{}X_M$ that preserves the $A$-valued inner product.
\end{proposition}

\begin{proof}
Let us fix bases $(\xi_i)_i \subset X$ and $(\eta_j)_j \subset Y$, localized in finite sets $[-M,-a]$ and $[-L, -1]$.
Consider the $A_{[-M,-1] \cup [0, M]}$-bimodule
\[
Z = X_{[-M,-1]} \boxtimes_{A_{[-M,-1]}} Y_{[-M,-1]} \otimes \overline{Y_{[0,M]}}.
\]
Note that the vectors $\xi_i \otimes \eta_j \otimes \overline{\eta^+}$ generate $Z$ as a right $A_{[-M,-1] \cup [0, M]}$-module.

There is a surjective right $A_{[-M,-1] \cup [0, M]}$-homomorphism $f_1\colon Z \to X^Y_M$ given by
\[
f_1(\xi^- \otimes \eta^- \otimes \overline{\eta^+}) = \xi^- \braket{ \eta^+ | \eta^- }.
\]
To be precise, we write $\xi^-$ as a linear combination $\sum_i \xi_i a_i$ for $a_i \in A_{[-M, -1]}$ and send $\xi^- \otimes \eta^- \otimes \overline{\eta^+}$ to $\sum_i \xi_i \braket{ \eta^+ | a_i \eta^- } \in X^Y_M$, which indeed agrees with $\xi^- \braket{ \eta^+ | \eta^- }$.

There is also a surjective $A_{[-M,-1] \cup [0, M]}$-homomorphism $f_2\colon Z \to \prescript{Y}{} X_M$ given by
\[
f_2(\xi^- \otimes \eta^- \otimes \overline{\eta^+}) = \braket{ \eta^+ | \eta^{-\prime} } \xi^{-\prime},
\]
where $\eta^{-\prime} \otimes \xi^{-\prime}$ denotes the image of $\xi^- \otimes \eta^-$ under the braiding $X \boxtimes_A Y \to Y \boxtimes_A X$, which restricts to a map $X_{[-M,-1]} \boxtimes_{A_{[-M,-1]}} Y_{[-M,-1]} \to Y_{[-M,-1]} \boxtimes_{A_{[-M,-1]}} X_{[-M,-1]}$.

Then, with $\xi^- = \xi_i$ and $\eta^- = \eta_j$, we have $\eta^{-\prime} \otimes \xi^{-\prime} = \eta_j \otimes \xi_i$.
Let us compare the inner products between the vectors $f_k(\xi_i \otimes \eta_j \otimes \overline{\eta^+})$.

For $k = 2$, we get
\[
\braket{ \braket{ \eta^+_2 | \eta_{j_2} } \xi_{i_2} | \braket{ \eta^+_1 | \eta_{j_1} } \xi_{i_1} } = \braket{ \xi_{i_2} | \braket{ \eta_{i_2} | \eta^+_2 } \braket{ \eta^+_1 | \eta_{j_1} } \xi_{i_1} }.
\]
Since the vectors $\eta_j$ and $\eta^+$ are localized to the right of $\xi_i$, the element $\braket{ \eta_{i_2} | \eta^+_2 } \braket{ \eta^+_1 | \eta_{j_1} }$ commutes with $\xi_{i_1}$, and the above is equal to $\braket{ \xi_{i_2} | \xi_{i_1} } \braket{ \eta_{i_2} | \eta^+_2 } \braket{ \eta^+_1 | \eta_{j_1} }$.

For $k = 1$, we get
\[
\braket{ \xi_{i_2} \braket{\eta^+_2 | \eta_{j_2} } | \xi_{i_1} \braket{ \eta^+_1 | \eta_{j_1} } } = \braket{ \eta_{i_2} | \eta^+_2 } \braket{ \xi_{i_2} | \xi_{i_1} } \braket{ \eta^+_1 | \eta_{j_1} }.
\]
Again by the disjointedness of the support, this is equal to $\braket{ \xi_{i_2} | \xi_{i_1} } \braket{ \eta_{i_2} | \eta^+_2 } \braket{ \eta^+_1 | \eta_{j_1} }$, hence we get the equality of inner products.

By the (right) $A_{[-M,-1] \cup [0, M]}$-linearity, we get
\[
\braket{ f_1(\xi^-_2 \otimes \eta^-_2 \otimes \overline{\eta^+_2}) | f_1(\xi^-_1 \otimes \eta^-_1 \otimes \overline{\eta^+_1}) } = \braket{ f_2(\xi^-_2 \otimes \eta^-_2 \otimes \overline{\eta^+_2}) | f_2(\xi^-_1 \otimes \eta^-_1 \otimes \overline{\eta^+_1}) }
\]
for all vectors $\xi^-_i$, $\eta^-_i$, and $\eta^+_i$.
\end{proof}

\begin{corollary}\label{cor:localized-left-basis-for-DHR}
Suppose that $A_\bullet$ satisfies Condition~\ref{cond:reduced-net}.
Then there is a left basis of $X$ localized in $[-2L-a, -a]$.
\end{corollary}

\begin{proof}
By assumption, the module $X^Y_M$ in Proposition~\ref{prop:use-braid-switch-from-left-mod-to-right-mod} (for big enough $Y$) agrees with $X_{[-2L,-1]} A_{[-M, M]} = X_{[-M, M]}$.
By this proposition, the subspace $A_{[-M, M]} X_{[-2L,-1]} = \prescript{Y}{}X_M \subset X_{[-M, M]}$ has the same dimension, hence they must agree.
We obtain the claim by letting $M \to \infty$ so that $[-M, M]$ covers $S$ from Proposition~\ref{prop:big-S-has-left-basis}.
\end{proof}

\begin{corollary}
Suppose that $A_\bullet$ satisfies Condition~\ref{cond:reduced-net-arbitrary-cut-point}.
If $X$ is an object of $\DHR(A_\bullet)$, then its dual bimodule $\overline{X}$ is again an object of $\DHR(A_\bullet)$.
\end{corollary}



\end{document}